\documentclass[12pt,a4paper,reqno]{amsart} 
\usepackage[margin=1 in]{geometry}
\usepackage{setspace}

\usepackage[utf8]{inputenc}
\usepackage[titletoc]{appendix}
\usepackage{amsfonts,amssymb,amscd,graphicx, amsmath,hyperref,subcaption,lscape}
\usepackage{fancybox}
\usepackage{color,float} 
\usepackage{bbm}
\usepackage{cite} 
\usepackage{wrapfig}
\allowdisplaybreaks
\usepackage{dsfont}
\usepackage[english]{babel}

\usepackage{tikz, tikz-cd}
\usetikzlibrary{snakes}
\usetikzlibrary{math}
\usepackage{tikz}
\usetikzlibrary{positioning}
\usepackage{setspace}
\usepackage{caption}
\usepackage[pagewise]{lineno}
\newcommand{\rvline}{\hspace*{-\arraycolsep}\vline\hspace*{-\arraycolsep}}
\usepackage{mathrsfs}
\usepackage[shortlabels]{enumitem}
\usepackage{xcolor}
\usepackage{tikz}
\mathchardef\ordinarycolon\mathcode`\:
\mathcode`\:=\string"8000
\begingroup \catcode`\:=\active
\gdef:{\mathrel{\mathop\ordinarycolon}}
\endgroup
\makeatletter
\renewcommand*\env@matrix[1][\arraystretch]{%
	\edef\arraystretch{#1}%
	\hskip -\arraycolsep
	\let\@ifnextchar\new@ifnextchar
	\array{*\c@MaxMatrixCols c}}
\makeatother

\usepackage{graphicx}
\usepackage{amsthm}
\usepackage{rotating}
\newtheorem{theorem}{Theorem}[section]
\newtheorem{proposition}[theorem]{Proposition}

\newtheorem*{notation}{Notation}
\newtheorem{corollary}{Corollary}[theorem]

\theoremstyle{definition}
\newtheorem{definition}[theorem]{Definition}
\newtheorem{remark}[theorem]{Remark}
\newtheorem{exam}[theorem]{Example}
\newcommand{\F}{\mathcal{F}}
\newcommand{\G}{\mathcal{G}}
\newcommand{\R}{\mathcal{R}}
\newcommand{\V}{\mathcal{V}}
\newcommand{\E}{\mathcal{E}}
\DeclareMathOperator{\lcm}{lcm}
\allowdisplaybreaks
\title[]{A combinatorial approach to study subshifts associated with multigraphs}
\author[Nikita Agarwal]{Nikita Agarwal}
\address{Department of Mathematics\\
	Indian Institute of Science Education and Research Bhopal\\
	Bhopal Bypass Road, Bhauri \\
	Bhopal 462 066, Madhya Pradesh\\
	India}
\email{nagarwal@iiserb.ac.in}
\author[Haritha Cheriyath]{Haritha Cheriyath}
\address{School of Mathematics\\
	Tata Institute of Fundamental Research\\
	Mumbai 400 005, Maharashtra\\
	India}
\email{harithacheriyath@gmail.com}
\author[Sharvari Neetin Tikekar]{Sharvari Neetin Tikekar}
\address{School of Mathematics\\
	Tata Institute of Fundamental Research\\
	Mumbai 400 005, Maharashtra\\
	India}
\email{sharvari.tikekar@gmail.com}

\date{\today}

\begin{document} 
	\maketitle

	\begin{abstract} 
		A subshift of finite type over finitely many symbols can be described as a collection of all infinite walks on a digraph with at most a single edge from a vertex to another. The associated finite set $\F$ of forbidden words is a constraint which determines the language of the shift entirely. In this paper, in order to describe infinite walks on a multigraph, we introduce the notion of multiplicity of a word (finite walk) and define repeated words as those having multiplicity at least $2$. In general, for given collections $\F$ of forbidden words and $\R$ of repeated words with pre-assigned multiplicities, we define notion of a generalized language which is a multiset. We obtain a subshift associated with $\F$ and $\R$ such that its entropy is calculated using the generalized language. We also study the relationship between the language of this subshift and the generalized language. We then obtain a combinatorial expression for the generating function that enumerates the number of words of fixed length in this generalized language. This gives the Perron root and eigenvectors of the adjacency matrix with integer entries associated to the underlying multigraph. Using this, the topological entropy and an alternate definition of Parry measure for the associated edge shift are obtained. We also discuss some properties of Markov measures on this subshift.
	\end{abstract}
	
	\noindent \textbf{Keywords}: Subshift of finite type, edge shift, Perron root and eigenvectors, Markov measure, correlation polynomial, generating function\\
	\noindent \textbf{2020 Mathematics Subject Classification}: 37B10 (Primary); 68R15, 05A16 (Secondary)
	
	\section{Introduction}
	Symbolic dynamics is mainly used as a tool to study dynamical systems that exhibit some kind of hyperbolicity such as Anosov diffeomorphisms or Axiom A maps~\cite{MH,Ruelle_book,Smale,Hadamard}. It involves studying the collection of infinite sequences on a given set of symbols known as a \textit{shift space} and is quite combinatorial in nature. A shift space which is characterized in terms of a finite collection of words that are forbidden, is popularly known as a \textit{subshift of finite type}. For instance, the collection of all infinite paths on a given digraph describes a subshift of finite type. The complexity of this subshift is given by the connectivity of this graph, which is therefore, given by the spectral properties of the associated adjacency matrix. When a graph does not have multiple edges from a vertex to another, any path can be represented in terms of its vertices as well as its edges. Such shifts are, in particular, known as the \textit{vertex} and \textit{edge} shifts, respectively. Shift spaces commonly appear, albeit with different terminologies, in other areas of sciences as well. For instance, in coding and information theory, the finite collection of words that are forbidden is called a \textit{constraint} and the corresponding subshift is known as the \textit{constrained coding}, refer~\cite{Coding_book, semiconstrainedTom, semocontrainedTom2,MR2032296}.
		
	A directed graph containing multiple edges between some pair of initial and terminal vertices is known as a {\it multigraph} and its associated adjacency matrix has non-negative integer entries. In this case, two distinct paths may have identical representation in terms of vertices. Hence, to avoid such a discrepancy, distinct labels are assigned to each edge. Each path is then described in terms of labelled edges and associated edge shifts are considered. Some dynamical properties of the edge shifts are studied in~\cite{Lind,Williams}. For instance, Williams~\cite{Williams}, using the concept of Markov partitions, models a class of dynamical systems as edge shifts associated with non-negative integer matrices via conjugacy. They also provide a complete characterization of these systems in terms of their adjacency matrices. Although any edge shift can be thought of as a vertex shift with respect to some other digraph, it is more convenient and useful to consider edge shift since it makes the associated adjacency matrices
	closed under multiplication. It also often reduces the size of the adjacency matrix significantly, but at the expense of adding multiple edges. Hence the study of edge shifts is important to understand several other complex dynamical systems. 
	
	An important invariant of shift spaces under conjugacy is {\it topological entropy}, which measures the growth rate of the number of words of finite length in the shift, and can be computed for a wide class of shifts. Computing entropy also gives rise to some interesting combinatorial problems, as it involves counting the number of finite words. However, there is not much literature on counting the finite words in an edge shift, especially when the edge shift is associated to some multigraph. This constitutes one of the problems that we aim to address in this work. Recent studies have proved that multigraphs provide a more suitable and efficient framework to model real life systems, including financial risk models~\cite{applicationFRM} and software architecture~\cite{onegraph_software}. A simple example would be of a world wide web multigraph model, where a vertex represents a webpage and a directed edge represents a hyperlink (arrowhead matrices). A certain class of directed multigraphs associated with arrowhed matrices are studied in~\cite{hubdirectedmulti}. 
	
	

	In this paper, we focus on multigraphs and present an alternate interpretation of edge shifts as follows. Instead of assigning distinct labels to all the directed edges between a fixed pair of initial and terminal vertices in a multigraph, we consider only a single edge and associate a number with it, called as {\it multiplicity} of an edge. The multiplicity represents the number of all the edges between the fixed pair of initial and terminal vertices. In this way, each edge can be represented uniquely by its initial and terminal vertices. 
	Since we keep track of the multiple occurrence of each edge through its multiplicity, information on the total number of paths is preserved. At the same time, the collection of symbols is reduced which makes the computations easier. 
	
	One of the fundamental results in the study of spectral properties of matrices is the well-known \textit{Perron Frobenius theorem}. It guarantees the existence of a maximal real simple eigenvalue of any irreducible non-negative matrix. This result has plenty of applications in various branches of mathematics, as well as in other disciplines such as network theory and engineering, see~\cite{Brin,Dembele21,Hawkins_book21,Keyfitz_book,NM,Parry64,Penner,BCL2021}. The Perron Frobenius theorem for matrices is useful in studying spectral and thermodynamic properties of shift spaces with associated binary adjacency matrix (a real matrix with each entry either 0 or 1). Guibas and Odlyzko~\cite{Guibas} give formulas for enumeration of words of given length over some symbol set which avoid certain set of words in terms of correlation polynomials of these words. Using these combinatorial techniques, an expression for the Perron root (the largest eigenvalue) and Perron eigenvectors of the binary adjacency matrix are obtained in~\cite{Parry}. These results have several independent applications in other areas including symbolic dynamics, game theory, information theory and network theory, some of which are given in~\cite{ABBG18,ABR20,Perron1,Perron_book}. Building upon the work of Guibas and Odlyzko, in this paper, we develop similar combinatorial tools by considering a modified set of rules in terms of multiple occurrence and non-occurrence of certain words. We then employ these results to study the spectral and thermodynamic properties of edge shifts associated with a non-negative integer matrix. 
	
	\subsection*{Organization of the paper}
	
	In Section~\ref{sec:SFT_binary}, we discuss some preliminaries on subshifts of finite type (SFT) associated to a binary adjacency matrix and some known results in this setup. In Section~\ref{sec:SFT_general}, we consider a collection $\F$ of forbidden words and $\R$ of repeated words having multiplicities at least $2$, with symbols from a finite set $\Sigma$. We introduce the concept of \emph{generalized language} denoted by $\Lambda^{p}$ associated to $\F$ and $\R$ where multiplicity of any word is computed in terms of multiplicities of words from $\R$. Here $p$ is the length of the longest word in $\F \cup \R$. We derive a necessary and sufficient condition on $\F$ and $\R$ such that $\Lambda^p$ becomes the language of an edge shift associated to some non-negative integer matrix. Moreover, in case when $\Lambda^{p}$ is not a language, there exists a non-negative integer matrix such that the language of the edge shift associated to this matrix is maximally contained in $\Lambda^{p}$. Interestingly, the entropy of this edge shift can be computed in terms of the number of words of length $n$, say $f(n)$, in $\Lambda^{p}$ with multiplicities.
	
	In Section~\ref{sec:recurrence relations}, a formula for generating function $F(z)$ of $f(n)$ is obtained. It is observed that the Perron root of the adjacency matrix associated to $\F$ and $\R$ can be expressed as the largest real pole of $F(z)$. In Section~\ref{sec:Perron EVectors}, we obtain an expression for the left and right Perron eigenvectors of the adjacency matrix associated to $\F$ and $\R$ in terms of the correlation of words from $\F\cup\R$. In Section~\ref{sec:normalization}, we derive a simple expression for the normalization factor for the Perron eigenvectors. The theory has been well-studied in~\cite{Combinatorial, Parry}, when the collection $\R$ is empty, that is, when there are no repeated words, or equivalently, no multiple edges in the associated graph. 
	
	In Section~\ref{sec:conjugacy and measures}, we observe that for a fixed $\F$, the SFT associated to $\F$ is a factor of the edge shift associated to $\F$ and $\R$ for any given collection $\R$. We then study the relation between the Markov measures on the respective shifts and find that the Parry measure on the SFT associated to $\F$ can be obtained as a push forward of the Shannon-Parry measure on the edge shift associated to $\F$ and $\R$ for some collection $\R$. 
	Section~\ref{sec:conclusion} consists of some concluding remarks and future directions of study. 
	%

	\section{Preliminaries}\label{sec:SFT_binary}
	Let $\Sigma$ be a finite set of symbols and $\Sigma^\mathbb{N}$ be the set of all one-sided sequences with symbols from $\Sigma$. A finite sequence with symbols from $\Sigma$ is called a \textit{word}. For a word $w$, let $|w|$ denote its length. A finite collection of words is said to be \textit{reduced} if for any two distinct words in the collection, one is not a subword of the other. Let $\mathcal{F}$ be a finite reduced collection of words with symbols from $\Sigma$ having length at least 2. Define the \textit{subshift of finite type}, denoted as $\Sigma_\F$, consisting of one-sided sequences which do not contain any word from $\F$ as a subword. A word is said to be \textit{allowed} if it is a subword of some sequence in $\Sigma_\F$. A word is said to be {\it forbidden} if it is not allowed in $\Sigma_{\F}$. 
	
	\begin{definition}[Language]\label{def:lang}
		The \textit{language of the subshift $\Sigma_\F$} is defined as the collection $\mathcal{L}=\bigcup\limits_{n\ge 1}\mathcal{L}_n$ of words, where for each $n\ge 1$, $\mathcal{L}_n$ denotes the collection of all allowed words of length $n$ in $\Sigma_\F$. 
	\end{definition}
	
	A given subshift of finite type can equivalently be thought as the collection of infinite paths on certain digraph. Hence a given SFT corresponds to a matrix which is the adjacency matrix associated to this digraph. We will soon discuss this correspondence.  For now, in the following subsection, we will focus on subshifts of finite type corresponding to binary adjacency matrices. The subshifts of finite type corresponding to general (non-negative integer) adjacency matrices will be discussed in the next section.
	
	\subsection{A subshift of finite type associated to a binary matrix}
	Let $\Sigma$ be a finite set of symbols. Let $A$ be a binary matrix whose rows and columns are indexed by $\Sigma$. Also, $A$ is assumed to be irreducible (a square non-negative matrix $A=(A_{xy})$ is called \emph{irreducible} if for every $x, y$, there exists $m = m(x,y)$ such that $A^{m}_{xy}>0$). Consider the digraph $\G_A=(\V_A,\E_A)$ associated to $A$ defined as $\V_A=\Sigma$ and $\E_A=\{(xy):A_{xy}=1\}$. An infinite path of vertices on $\G_A$ is of the form $x_1x_2\dots$ where $A_{x_ix_{i+1}}=1$ for all $i$. Similarly, an infinite path of edges on $\G_A$ is of the form $(x_1y_1)(x_2y_2)\dots$ where $A_{x_iy_i}=1$ and $y_i=x_{i+1}$ for all $i$. 
	
	\begin{definition}[Vertex and Edge Shifts]
		The collection of all infinite paths of vertices on $\G_A$ is called as the \emph{vertex shift} associated to $A$ and is denoted as $\Sigma^v_A$. That is,
		\[
		\Sigma^v_A=\left\{x_{1}x_{2}x_{3} \dots \ : \  x_{i} \in \mathcal{V}_{A}, \ A_{x_{i}x_{i+1}}=1\right\}.
		\]
		The \emph{edge shift} associated to $A$, denoted by $\Sigma_A$, is defined as the collection of all infinite paths of edges on $\G_A$, that is,
		\[
		\Sigma_A=\left\{(x_{1}x_{2}) (x_{2}x_{3}) \dots \ : \  x_{i} \in \mathcal{V}_{A}, \ A_{x_{i}x_{i+1}}=1\right\}.
		\]
	\end{definition}
	\noindent Note that $\Sigma^v_A=\Sigma_\F$ where $\F=\{xy:A_{xy}=0\}.$ Here all words in $\F$ are of length 2. Further, the number of allowed words of length $n$ in $\Sigma_{\F}$, given by $|\mathcal{L}_n|$, equals the sum of entries of $A^{n-1}$. Also, $\Sigma^v_A$ and $\Sigma_A$ are conjugates where the conjugacy \footnote{ Here conjugacy is simply a ``recoding" of sequences (we refer to \cite[Definition 1.5.9]{LM_book}). In symbolic dynamics, this type of conjugacy is known as \emph{sliding block code}.} is given by $x_1x_2x_3\dots \longmapsto (x_1x_2)(x_2x_3)\dots$. 
	
	In short, for a given binary matrix, both vertex and edge shifts are conjugate to an SFT where the forbidden collection consists of words of length 2. Conversely, if we are given a forbidden collection $\F$ consisting of words of length 2, we can obtain a digraph such that the vertex shift and the edge shift associated to it are conjugate to $\Sigma_\F$. In the following section, we consider a forbidden collection $\F$ that contains words of arbitrary lengths and analyze properties similar to those just discussed.  
	
	\subsection{A binary adjacency matrix for a general $\F$}
	Let $\F$ be a finite reduced collection of forbidden words with symbols from $\Sigma$. Let $\Sigma_\F$ be the associated SFT. Throughout this paper, \textit{we assume that symbols in $\Sigma$ are not forbidden}, otherwise we can remove them from the collection of symbols $\Sigma$. Let $\mathcal{L}=\bigcup_{n\ge 1}\mathcal{L}_n$ be its language. 
	For $m\ge 1$, we define a conjugacy from $\Sigma_\F$ to an $m$-step subshift $\Sigma_\F^{[m]}$ as 
	$$x_1x_2x_3\dots \longmapsto (x_1\dots x_{m})(x_2\dots x_{m+1})(x_3\dots x_{m+2})\dots, $$ where the subshift $\Sigma_{\F}^{[m]}$ has symbols from $\mathcal{L}_{m}$ (refer higher block presentation in~\cite{LM_book}). 
	
	\begin{definition}\label{def:star}
		Let $m\ge 2$ and $X=x_1x_2\dots x_m,Y=y_1y_2\dots y_m\in\mathcal{L}_m$. Then $X*Y$ is defined only if $x_2\dots x_m=y_1\dots y_{m-1}$, and
		\[
		X*Y:=x_1x_2\dots x_my_m.
		\] 
		If $X=x_1, Y=y_1\in \mathcal{L}_{1}$, then $X*Y=x_1y_1$.
	\end{definition}

	Let $p$ be the length of the longest word in $\F$. 
	Now we define an \emph{adjacency matrix} $A=(A_{XY})$ indexed by words from $\mathcal{L}_{p-1}$ (the indexing of rows and columns are in lexicographic order for convenience), where
	\[A_{XY}=
	\begin{cases}
		
		1, & \text{if } X*Y \text{ is defined and } X*Y\in\mathcal{L}_p\\
		0, & \text{ otherwise}.
	\end{cases}
	\]
	This adjacency matrix gives a digraph $\G_A=(\V_A,\E_A)$ where $\V_A=\mathcal{L}_{p-1}$ and $\E_A=\{(XY):A_{XY}=1\}$. We label $(XY)\in\E_A$ as $X*Y$. Then the vertex shift $\Sigma^v_A$ consists of infinite paths of vertices on $\G_A$ that are of the form $X_1X_2X_3\dots$, where $A_{X_iX_{i+1}}=1$ for all $i$. That is $\Sigma^v_A=\Sigma_\F^{[p-1]}$. Similarly, the edge shift $\Sigma_A= \Sigma_\F^{[p]}$. Note that $\Sigma^v_A$ and $\Sigma_A$ are conjugates to $\Sigma_\F$. Since the words from $\mathcal{L}_{p-1}$ and $\mathcal{L}_{p}$ completely describe the subshift, the words in $\mathcal{L}_{p-1}\cup \mathcal{L}_{p}$ can be regarded as the building blocks of $\Sigma_{\F}$. Moreover, $|\mathcal{L}_n|$ is given by the sum of entries of $A^{n-p+1}$. 
	A vertex  (edge) shift is said to be {\it irreducible}, if the adjacency matrix associated to it is irreducible. In this case, the graph $\G_{A}$ associated to $A$ is strongly connected. 
	
	\subsection{Subshift of finite type as a dynamical system}
	For a given $\F$ with longest word of length $p$, consider the adjacency matrix $A$ indexed by words from $\mathcal{L}_{p-1}$, as defined in the previous subsection. In this subsection, we discuss some properties of $\Sigma^v_A$ from a dynamical viewpoint. Since $\Sigma_{\F}$ or $\Sigma_A$ (in fact any $\Sigma_\F^{[m]}$, $m\ge 1$) is conjugate to $\Sigma^v_A$, they exhibit the same properties.
	
	The \emph{left shift map} $\sigma:\Sigma^v_A\rightarrow\Sigma^v_A$ is defined as $\sigma(X_1X_2X_3\dots)=X_2X_3\dots$. A $\sigma$-invariant probability measure $\mu$ on $\Sigma^v_A$, known as the \emph{Parry measure}, is defined as follows. 
	
	\begin{definition}[The Parry measure]
		For an allowed word $w=X_1X_2\dots X_n$, let $C_w$ denote the cylinder based at $w$, that is, $C_w$ consists of all sequences in $\Sigma^v_A$ that begin with the word $w$. Then
		\[
		\mu(C_w)=\frac{U_{X_1}V_{X_n}}{\theta^{n-1}},
		\] where $\theta$ is the largest eigenvalue of $A$ (known as the Perron root), which exists, and is positive by the Perron Frobenius theorem. The vectors $U$ and $V$, with indexing of rows same as that of $A$, are the left and right Perron (column) eigenvectors corresponding to the Perron root satisfying $U^TV=1$. Extend $\mu$ to the sigma-algebra generated by all cylinders. Since $\Sigma^v_A$ and $\Sigma_A$ are conjugates, this measure can be extended to give a measure on $\Sigma_A$, which will also be called the Parry measure. 
	\end{definition}

	\noindent
	It is a fact that the topological entropy of $\Sigma_\F$ (or $\Sigma^v_A$ or $\Sigma_A$) which is defined as \[
	h_{\text{top}}(\Sigma_\F)=\lim_{n\to\infty}\frac{1}{n}\ln(|\mathcal{L}_n|),\]
	is given by $\ln(\theta)$, where $\theta$ is the Perron root of $A$. The Parry measure is the unique measure of maximal entropy according to the variational principle, that is, the measure theoretic entropy with respect to $\mu$ is equal to $\ln(\theta)$, see~\cite{Walters_book}. In~\cite{Parry}, the authors obtain a combinatorial expression for the Perron root and the associated eigenvectors of $A$, for a given forbidden collection $\F$. We now recall their results. 
	
	\subsection{Summary of existing results}	
	The concept of correlation between two words which determines the overlap between them, is described below.
	%
	\begin{definition}[Correlation]\label{def:corr}
		The \emph{correlation} of two words $u$ and $v$, denoted by $(u, v)$, is a binary string $(c_{1},\dots,\,c_{|u|})$ of length $|u|$, defined by the following algorithm. The $i^{th}$ bit $c_{i}$ is determined by placing $v$ below $u$ such that the leftmost symbol in $v$ lies under the $i^{th}$ symbol from left in $u$. Set $c_{i} = 1$ if and only if the overlapping segments of $u$ and $v$ are identical, else set $c_{i} = 0$.
		The correlation of $u$ and $v$ can also be interpreted as a polynomial in some variable $z$ as, 
		$(u,v)_{z}=\sum\limits_{i=1}^{|u|} c_{i} z^{|u|-i}.$ The polynomial $(u,v)_z$ is termed as the \emph{correlation polynomial}.
	\end{definition}
	
	Let $\F=\{a_1,a_2,\dots,a_s\}$ be a reduced collection of forbidden words with symbols from $\Sigma$, $\Sigma_\F$ be the associated subshift and $\mathcal{L}=\bigcup_{n\ge 1}\mathcal{L}_n$ be its language. Let $f(n)=|\mathcal{L}_n|$ and $f_i(n)$ be the number of words of length $n$ ending with $a_i$ with symbols from $\Sigma$ not containing any of the words from $\F$ except the single appearance of $a_i$ at the end. Let $F(z)=\sum\limits_{n=0}^\infty f(n)z^{-n}$ and $F_i(z)=\sum\limits_{n=0}^\infty f_i(n)z^{-n},\ 1\le i\le s,$ be their respective generating functions. 
	The following result gives an expression for these generating functions in terms of the correlation between the forbidden words. 
	
	\begin{theorem}~\cite[Theorem 1]{Combinatorial}\label{thm:Guibas}
		The generating functions $F(z)$, $F_i(z)$ satisfy the linear system of equations 
		\begin{eqnarray*}
			K(z)\begin{pmatrix}F(z) \\ F_1(z) \\ \vdots \\ F_s(z)\end{pmatrix} = \begin{pmatrix}z \\ 0 \\ \vdots \\ 0\end{pmatrix},
		\end{eqnarray*}
		where $K(z)=\begin{pmatrix}z-q & z \mathbbm{1}^T \\ \mathbbm{1} &-z\mathcal{M}(z)\end{pmatrix}$, $\mathcal{M}(z)=((a_j,a_i)_z)_{1\le i,j\le s}$ is the correlation matrix for the collection $\F$, $\mathbbm{1}$ denotes the column vector of size $s$ with all 1's and $q$ is the size of $\Sigma$. In particular, \[F(z)=\frac{z}{z-q+R(z)},\]
		where $R(z)$ is the sum of the entries of $\mathcal{M}^{-1}(z)$.
	\end{theorem}
	In~\cite[Theorem 1]{Parry}, the Perron root of the associated adjacency matrix $A$ is proved to be the largest real pole of $F(z)$ and hence we have the following result.
	
	\begin{theorem}
		With the notations as above, the Perron root $\theta>0$ is the largest (in modulus) zero of the rational function  $z - q + R(z)$.
	\end{theorem}
	
	In~\cite{Parry}, the authors also obtain an expression for the left and right Perron eigenvectors corresponding to the Perron root, in terms of the correlation polynomials. This is then used to obtain a combinatorial expression for the Parry measure on $\Sigma_\F$. 	
	
	\section{Subshift associated with a non-negative integer matrix}\label{sec:SFT_general}
	In the preceding section, we obtained an SFT associated to a given binary matrix and vice-a-versa. In this section, we generalize this concept to a subshift of finite type associated to a non-negative integer matrix. Here the corresponding digraph may have multiple edges between its vertices. 
	
	\subsection{A subshift of finite type from a non-negative integer matrix}\label{subsec:general_matrix}
	Let $\Sigma$ be a finite symbol set.  
	Let $A=(A_{xy})$ be a non-negative integer matrix indexed by $\Sigma$. Consider the directed multigraph $\mathcal{G}_{A} = (\mathcal{V}_{A}, \mathcal{E}_{A})$ associated to the matrix $A$ where $\mathcal{V}_{A} = \Sigma$ and the number of edges from the vertex $x$ to $y$ is given by $A_{xy}$. Let us visualize $\G_A$ labelled with all the edges of initial vertex $x$ and terminal vertex $y$ distinctly as, $(xy)_{1}, \dots, (xy)_{A_{xy}}$. 
	Using the edge set $\mathcal{E}_A$ as a new symbol set, we consider the edge shift associated with $A$, denoted by $\Sigma_A$, as
	\[
	\Sigma_A=\left\{(x_{1}x_{2})_{i_{1}} (x_{2}x_{3})_{i_{2}} \dots \ : \  x_{k} \in \mathcal{V}_{A}, \ 1 \le i_{k} \le A_{x_{k}x_{k+1}},\ k\ge 1\right\}.
	\]
	Let $\mathfrak{L}_n$ be the set of all allowed words of length $n$ in $\Sigma_A$, and let $\mathfrak{L} = \bigcup\limits_{n \, \ge \, 1} \mathfrak{L}_{n}$ be the language of $\Sigma_A$. Note here that the label for the language is different than that in Definition~\ref{def:lang} .
	
	\begin{definition}[Shannon-Parry measure]\label{def:sp}
		The \emph{Shannon-Parry measure} (refer to~\cite{Shannon,LM_book}) on $\Sigma_A$ is defined as
		\[
		\mu(C_W)=\frac{U_{x_1}V_{x_{n+1}}}{\theta^n},
		\] where $C_W$ is the cylinder based at $W=(x_{1}x_{2})_{i_{1}} (x_{2}x_{3})_{i_{2}} \dots (x_{n}x_{n+1})_{i_{n}}\in\mathfrak{L}_n$, $\theta$ is the Perron root of $A$ and $U$ and $V$ are the left and right Perron (column) eigenvectors such that $U^TV=1$. 
	\end{definition}
	
	The Shannon-Parry measure $\mu$ is invariant with respect to the left shift map $\sigma$ on $\Sigma_A$. It is the unique measure of maximal entropy where the topological entropy of $\Sigma_A$ is given by $h_{\text{top}}(\Sigma_A)=\lim_{n\to\infty}\frac{1}{n}\ln(|\mathfrak{L}_n|)$ and is the same as $\ln(\theta)$.  
	
	Let $\F=\{xy:A_{xy}=0\}$, $\Sigma_\F$ be the associated SFT and $\mathcal{L}=\bigcup_{n\ge 1}\mathcal{L}_n$ be its language. Note that $\Sigma_\F$ consists of all infinite paths of vertices on $\G_A$. However, unlike in Section~\ref{sec:SFT_binary}, $\Sigma_\F$ and $\Sigma_A$ need not be conjugates as $\Sigma_\F$ does not differentiate the multiple edges. To incorporate this, we define a multiplicity for each word in $\Sigma_\F$. 
	
	\begin{definition}[Multiplicity of an allowed word]\label{def:multi}
		Let $w=x_1x_2\dots x_n\in\mathcal{L}_n$. The \emph{multiplicity of $w$ in $\G_A$}, denoted as $m(w)$, is defined as the number of paths (of edges) of length $n-1$ in $\G_A$ with fixed vertices $x_1,\dots,x_{n}$ (in order). In other words, $m(w)=\prod_{i=1}^{n-1}A_{x_ix_{i+1}}$.
	\end{definition}
	
	\noindent
	A word $ w \in \mathcal{L}_n$ corresponds to $m(w)$ many words of the form $(x_{1}x_{2})_{i_{1}} (x_{2}x_{3})_{i_{2}} \dots (x_{n-1}x_{n})_{i_{n-1}}$ in $\mathfrak{L}_{n-1}$. Hence \[|\mathfrak{L}_{n-1}|= \sum_{w\in\mathcal{L}_{n}}m(w),\]
	which equals the sum of entries of the matrix $A^{n-1}$.  
	
	A word is called {\it repeated} if its multiplicity is greater than $1$.
	For $x,y \in \mathcal{V}_{A}$, $m(xy) = A_{xy}$. Define $\R$ to be the collection of repeated words given by $\R=\left\lbrace xy \, : \, A_{xy} >1 \right\rbrace$. Note that the multiplicity of any allowed word can be computed in terms of multiplicities of words from $\R$. Here words from $\F$ and $\R$ have length 2 and they uniquely determine $\Sigma_A$.

	\subsection{A non-negative integer matrix for general collections $\F$ and $\R$}
	Now we consider a general situation where a finite reduced collection $\F$ of forbidden words is given. Let $\Sigma_\F$ be the associated subshift of finite type. Assume that $\Sigma_\F$ is irreducible with language given by $\mathcal{L}=\bigcup_{n\ge 1}\mathcal{L}_n$, where $\mathcal{L}_n$ is the collection of allowed words of length $n$ in $\Sigma_\F$. 
	
	Let $\R=\{ r_{1}, r_{2}, \dots, r_{\ell}  \}$ be a reduced collection of allowed words in $\Sigma_\F$. 
	For each word $r_i\in\R$, we assign a number $m_i$, at least 2, which we call the \textit{multiplicity} of the word $r_i$.  
	For convenience of notation, we denote $\R$ as, $\R = \{ r_{1} (m_{1}), r_{2}(m_{2}), \dots, r_{\ell}(m_{\ell}) \}$, where the multiplicity $m_{i}$ of the word $r_{i}$ is mentioned in the corresponding parentheses. Here the words from $\F$ or $\R$ can be of arbitrary lengths. Given a collection $\R$, we extend the concept of multiplicity to all finite words as follows.
	\begin{definition}[Multiplicity of a word]
		Let $\R=\{r_1(m_1),\dots,r_\ell(m_\ell)\}$ be a given collection consisting of allowed words in $\Sigma_\F$.  For any $w\in\Sigma^n$, we define its \emph{multiplicity}, denoted by $m(w)$, as, 
		\[
		m(w) = \begin{cases}
			\prod_{r_i\in\R} m_i^{n(w ; r_i)}, & \text{ if } w\in\mathcal{L}_n,\\
			0, & \text{ otherwise } \\
		\end{cases}
		\]
		where $n(w;r_i)$ denotes the number of times the word $r_i\in\R$ appears as a subword of the word $w$. An allowed word is said to be \emph{repeated} if $m(w)>1$. 
	\end{definition}

	Clearly if an allowed word $w$ does not contain any repeated word then $m(w)$ equals 1. Moreover for each $r_i\in\R$, $m(r_i)=m_i$.

	
	\begin{definition}[Generalized Language]\label{def:lambda_n}
		For $n \ge 1$, define a multiset
		\[
		\Lambda_{n}:=\left\lbrace (w, m(w)) :w\in\mathcal{L}_n \right\rbrace,
		\] where we use the notation $(w, m(w))$ to denote the fact that an allowed word $w\in \mathcal{L}_n$ appears $m(w)$ many times in $\Lambda_n$. We define the \emph{generalized language corresponding to $\F$ and $\R$ of level $m\ge 1$} as the multiset,
		\[
		\Lambda^m=\bigcup_{n\ge m}\Lambda_n.
		\]
	\end{definition}
	
Using the properties of a multiset, the \textit{cardinality} of $\Lambda_{n}$, denoted by $|\Lambda_{n}|$, is given by, $|\Lambda_{n}| = \sum_{w \in \mathcal{L}_{n}} m(w)$. In what follows, the statement `$w\in\Lambda_n$' will mean that $w$ is an element of $\mathcal{L}_n$ and is repeated $m(w)$ many times in the multiset $\Lambda_n$.

	\begin{exam}
		Let $\Sigma=\{0,1\}$, $\F=\{01\}$ and $\R=\{00(\alpha), 111(\beta)\}$. Then $\Lambda_1=\{(0,1), (1,1)\}$, $\Lambda_2=\{(00,\alpha),(10,1),(11,1)\}$, $\Lambda_3=\{(000,\alpha^2),(100,\alpha),(110,1),(111,\beta)\}$.
	\end{exam}
	
	%
	
	\begin{remark}\label{rem:gen_matrix}
		Let $A=(A_{xy})$ be a non-negative integer matrix and $\G_A$ be the associated digraph. Consider the edge shift $\Sigma_A$ associated to $A$ as in Section~\ref{subsec:general_matrix} and let  $\F=\{xy:A_{xy}=0\}$ and $\R=\{xy:A_{xy}>1\}$ with $m(xy)=A_{xy}$. For $n\ge 2$ and an allowed word $w=x_1\dots x_n\in\Lambda_n$, $m(w)$ is the multiplicity of $w$ in $\G_A$ as given in Definition~\ref{def:multi}, that is, $m(w)$ is the number of paths on $\G_A$ with fixed vertices $x_1,x_2,\dots,x_{n}$ (in order). 
		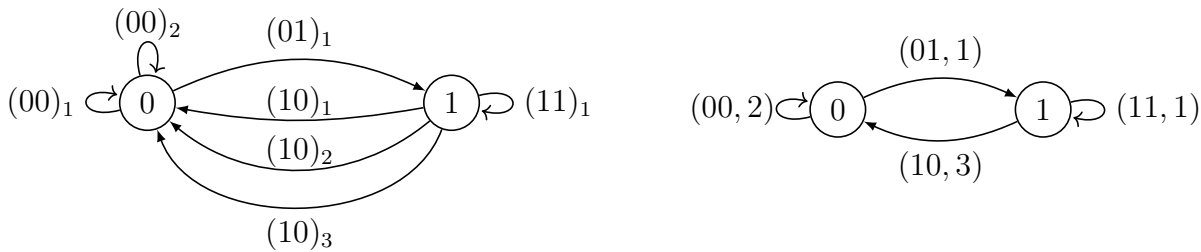
\begin{figure}[h]
			\centering
			$\displaystyle
			\begin {tikzpicture}[-latex ,auto ,node distance =3cm and 4cm ,on grid , semithick , state/.style ={draw, circle}] 
			\node[state,scale=1] (A) {$0$};
			\node[state,scale=1] (B) [right =of A] {$1$};
			\path (A) edge [loop left] node[right=-1.2cm] {$(00)_{1}$} (A);
			\path (A) edge [loop above] node[below=-0.6cm] {$(00)_{2}$} (A);
			\path (A) edge [bend left =25] node[above] {$(01)_{1}$} (B);
			\path (B) edge [bend right =-10] node[below =-.6 cm] {$(10)_{1}$}(A);
			\path (B) edge [bend right = -40] node[below =-.6 cm] {$(10)_{2}$}(A);		
			\path (B) edge [bend right =-70] node[below] {$(10)_{3}$} (A);
			\path (B) edge [loop right] node[right] {$(11)_{1}$} (B);
		\end{tikzpicture}
		\hspace{1cm}
		\begin {tikzpicture}[-latex, baseline=-2cm, auto ,node distance =2cm and 2.7cm ,on grid ,semithick ,state/.style ={draw, circle}] 
		\node[state] (A) {$0$};
		\node[state] (B) [right =of A] {$1$};
		\path (A) edge [loop left] node[right=-1.3cm] {$(00,2)$} (A);
		\path (A) edge [bend left =25] node[above] {$(01,1)$} (B);
		\path (B) edge [bend right=-25] node[below] {$(10,3)$}(A);
		\path (B) edge [loop right] node[right] {$(11,1)$} (B);
	\end{tikzpicture}
	$
	\caption{Standard interpretation on the left and new interpretation on the right}
	\label{fig:new_graph}
\end{figure}

\noindent
For $n\ge 2$, an allowed word $(x_{1}x_{2})_{i_{1}} (x_{2}x_{3})_{i_{2}} \dots (x_{n-1}x_{n})_{i_{n-1}}$ in $\Sigma_A$ corresponds to exactly one of $m(w)$ many repetitions of $w=x_1x_2\dots x_{n}$ giving a bijection between the language of $\Sigma_A$ and the generalized language $\Lambda^2$. Moreover, the topological entropy of $\Sigma_A$ is $h_{\text{top}}(\Sigma_A) = \lim_{n \to \infty} \frac{1}{n} \ln(|\Lambda_{n}|) $. Hence we can say that $\Lambda^2$ provides an alternate interpretation of $\Sigma_A$. An illustration of this interpretation is given in Figure~\ref{fig:new_graph} for the adjacency matrix $A=\begin{pmatrix}
	2&1\\3&1
\end{pmatrix}$. Note that the graph on the left is $\G_A$ with edges labelled.

\end{remark}

Now we address the converse of Remark~\ref{rem:gen_matrix}.
Let $p$ denote the length of the longest word in $\F \cup \R$. Our aim is to find a matrix $A$ (indexed by $\mathcal{L}_{p-1}$) such that the entropy of the edge shift $\Sigma_A$ is $h_{\text{top}}(\Sigma_A) = \lim_{n \to \infty} \frac{1}{n} \ln(|\Lambda_{n}|)$. As before, we will need to consider the higher block representation. One possibility for such a matrix $A$ is when the associated edge shift $\Sigma_A$ has language (in bijection with) $\Lambda^p=\bigcup_{n\ge p}\Lambda_{n}$. 

Using sliding block code representation, $\Lambda^p$ is the language of $\Sigma_A$ for some adjacency matrix $A$ indexed by $\mathcal{L}_{p-1}$ if and only if for each $w=x_1\dots x_n\in\Lambda_n$ with $n\ge p$, $m(w)$ matches with the number of paths with fixed vertices $X_1,\dots,X_{n-p+2}$ (in order) where $X_i=x_i\dots x_{i+p-2}$, for $1\le i\le n-p+2$.



%
%


\begin{exam}\label{ex:1} 
Let $\Sigma=\{0,1\}$, $\F=\{11\}$ and $\R=\{00(\alpha)\}$.
If $A=\begin{pmatrix}
\alpha & 1\\ 1 & 0
\end{pmatrix},$ then 
by Remark~\ref{rem:gen_matrix},  $\Sigma_A$ has language $\Lambda^2$. In fact, whenever words from $\F$ and $\R$ have length 2, then $\Sigma_A$ has language $\Lambda^2$ where $A_{xy}=m(xy)$. 
%
\end{exam}

\begin{remark}\label{rem:full_matrix}
For given collections $\F,\R$, a seemingly natural choice for the adjacency matrix $A$ is the one indexed by words from $\mathcal{L}_{p-1}$, with $A_{XY}=m(X*Y)$, for each $X,Y\in \mathcal{L}_{p-1}$ (as in Example~\ref{ex:1}). But the following example shows that this is not the right choice always. Let $\Sigma=\{0,1\}, \F=\{00\}$ and $\R=\{110(2),01(3)\}$. Here $p=3$, and the matrix $A$, indexed by $\{01,10,11\}$, is given by $\begin{pmatrix}
0&3&3\\3&0&0\\0&2&1
\end{pmatrix}$ and its associated graph $\G_{A}$ is as shown in Figure~\ref{fig:eg}.
\begin{figure}[h]
\centering
$\begin{tikzpicture}[-latex, baseline=-1.2cm, auto=left, node distance =2cm and 2.5cm ,on grid, semithick, state/.style ={draw, circle}]
\node[state,scale=0.8] (A) {$01$};
\node[state,scale=0.8] (B) [right =of A] {$10$};
\node[state,scale=0.8] (C) [below =of A] {$11$};
\path (C) edge [loop below] node[above] {} (C);
\path (A) edge [bend left =25] node[above] {} (B);
\path (A) edge [bend left =50] node[above] {} (B);
\path (A) edge [bend left =75] node[above] {} (B);
\path (A) edge [bend right =25] node[left] {} (C);
\path (A) edge [bend right =50] node[left] {} (C);
\path (A) edge [bend right =75] node[left] {} (C);
\path (B) edge [bend right =-25] node[below] {} (A);
\path (B) edge [bend right =-50] node[below] {} (A);
\path (B) edge [bend right =-75] node[below] {} (A);
\path (C) edge [bend left =-50] node[above] {} (B);
\path (C) edge [bend left =-75] node[above] {} (B);
\end{tikzpicture}
$
\caption{The graph $\G_{A}$}
\label{fig:eg}
\end{figure}

\noindent
Here $m(1010)=3$. However, in $\G_A$, there are 9 distinct paths of length 2 with fixed vertices $X_1=10,X_2=01,X_3=10$ (in order). Also, $h_{\text{top}}(\Sigma_A)\sim\ln(3.9)$, but $\lim_{n \to \infty}\frac{1}{n}\ln(|\Lambda_{n}|)\sim\ln(2.6)$ (this limit is calculated using the results from Section~\ref{sec:recurrence relations}).
\end{remark}


\noindent Let $X=x_1\dots x_n, Y=y_{1}\dots y_{n}$ such that $X*Y$ is defined but $X*Y\notin\R$. Then,
\begin{eqnarray*}
m(X* Y)&=&\dfrac{m(X)m(Y)}{m(x_{2}\dots x_{n})}.
\end{eqnarray*}
In order to incorporate this term in the denominator, for any word $v=v_1\dots v_n\in\mathcal{L}_n$ we define a number
\[
k(v):=\dfrac{m(v)}{m(v_2\dots v_{n})}. 
\] 
Therefore $m(X*Y)=k(X)m(Y)$, provided $X*Y\notin\R$.
It is easy to see that $k(v)>1$ if and only if $v$ begins with a word from $\R$.\\
For each $v\in\Lambda_n,n>p,$ the multiplicity of $v$ can be obtained from the multiplicities of words from $\Lambda_{p-1}\cup \Lambda_p$. Hence the words in $\Lambda_{p-1}\cup \Lambda_p$ act as building blocks for the words in $\Lambda_n$.\\ 
The following result provides a necessary and sufficient condition on $\F$ and $\R$ such that $\Lambda^p$ is a language of $\Sigma_A$ for some $A$.
\begin{theorem}\label{thm:lang}
Let $p$ be the length of the longest word in $\F\cup\R$. The generalized language $\Lambda^{p}=\bigcup_{n\ge p}\Lambda_n$ is the language of an edge shift if and only if all the words in $\R$ are of length $p$. 
\end{theorem}
\begin{proof}
Let $\mathcal{L}=\bigcup_{n\ge1}\mathcal{L}_n$ be the language of $\Sigma_\F$. Suppose $\R$ contains words of equal length $p$. Define a matrix $A$ indexed by $\mathcal{L}_{p-1}$ as follows,
\[A_{XY}=
\begin{cases}
m(X*Y), & \text{ if } X*Y\in \mathcal{L}_{p} \\
0, & \text{ otherwise }
\end{cases}
\]
Consider the associated digraph $\G_A$ and the edge shift $\Sigma_A$. We claim that $\Lambda^p$ is the language of $\Sigma_A$. That is, for each word $w=x_1x_2\dots x_n\in\Lambda_n, n\ge p$, $m(w)$ is same as the number of paths with fixed vertices $X_1,\dots,X_{n-p+2}$ (in order), where $X_i=x_i\dots x_{i+p-2}$, for $1\le i\le n-p+2$. Note that the number of paths in $\G_A$ with fixed vertices $X_1,X_2,\dots,X_n$ (in order) is 
\[
m(X_1*X_{2})m(X_{2}*X_{3})\dots m(X_{n-1}*X_{n}),
\]
which is the same as $m(w)$ since $\R$ has words of length $p$. Hence the claim.\\
Conversely, suppose $\Lambda^p$ is the language of an edge shift. Since $\Lambda_{p-1}\cup\Lambda_p$ is the building block for the words in $\Lambda^p$, the edge shift has to be associated with the matrix $A$ as defined in the previous part of the proof (since path of length 1 defines the entries in the matrix). We need to prove that all words in $\R$ are of length $p$. Suppose on the contrary, there exists $X=x_1\dots x_{p-1}\in\Lambda_{p-1}$, such that $m(X)>1$. Choose $x,y\in\Sigma$, such that $xx_1\dots x_{p-1}y\in\Lambda_{p+1}$ (use irreducibility of $\Sigma_\F$). The number of paths with fixed vertices $xx_1\dots x_{p-2},X,x_2\dots x_{p-1}y$ is $m(xx_1\dots x_{p-1})m(x_1\dots x_{p-1}y)$. However,
\[
m(xx_1\dots x_{p-1})m(x_1\dots x_{p-1}y)=m(xx_1\dots x_{p-1}y)m(X)>m(xx_1\dots x_{p-1}y),
\]
which is a contraction.
\end{proof}
By Theorem~\ref{thm:lang}, $\Lambda^p$ need not always be a language. The following result describes the structure of $\Lambda^p$ in general. Let $\mathcal{L}_1$ and $\mathcal{L}_2$ be languages of two subshifts of finite type. Then we write $\mathcal{L}_1\subseteq \mathcal{L}_2$ if $\mathcal{L}_1$ is contained in  $\mathcal{L}_2$ upon relabeling the symbols of $\mathcal{L}_1$.
\begin{theorem}\label{thm:same_roots}
With the notations as above, let $p$ be the length of the longest word in $\F\cup\R$. Then there exists a non-negative integer matrix $A$ indexed by $\mathcal{L}_{p-1}$ such that the associated edge shift $\Sigma_A$ with language $\mathcal{L}(\Sigma_A)$ satisfies the following.
\begin{enumerate}[label=(\roman*)]
\item $\mathcal{L}(\Sigma_A) \subseteq \Lambda^{p}$, \label{item:prop1}
\item The topological entropy $h_{top}(\Sigma_A)=\lim_{n\rightarrow \infty} \dfrac{1}{n}\ln |\Lambda_n|$. \label{item:prop2}
\item $h_{top}(\Sigma_A)>0$, if $\R\ne\emptyset$.
\end{enumerate}
\end{theorem}	
\begin{proof}
	Define a matrix $A=(A_{XY})_{X,Y\in \mathcal{L}_{p-1}}$ indexed by $\mathcal{L}_{p-1}$ as follows: for $X,Y\in\mathcal{L}_{p-1}$, let
	\begin{equation}\label{A_for_F_R}
		A_{XY}=\begin{cases}
			k(X*Y),&\text{ if } X*Y\in\mathcal{L}_p\\
			0,& \text{ otherwise}.
		\end{cases}
	\end{equation}
(i) Let $\G_A$ be the graph associated with the matrix $A$.
	Note that for a word $w=X*Y\in\Lambda_p$, the number of edges in $\G_A$ from vertex labelled $X$ to vertex labelled $Y$ equals $k(X*Y)$, and also $k(X*Y)\le m(w)$.  \\
	Let $n>p$. Let $w=x_1x_2\dots x_n\in\Lambda_n$ be an allowed word. Using induction on $n$, for $X_i=x_i\dots x_{i+p-2}$, we have
	\[
	m(w)=k(X_1*X_2)\dots k(X_{n-p}*X_{n-p+1})m(X_{n-p+1}*X_{n-p+2}).
	\]
	Also, the number of paths with vertices labelled $X_1,\dots, X_{n-p+2}$ (in order) in $\G_A$ is given by $k(X_1*X_2)\dots k(X_{n-p+1}*X_{n-p+2})$, which is at most $m(w)$. Hence, clearly $\mathcal{L}(\Sigma_A)\subseteq\Lambda^p$. \\
	(ii) Let $\tau(n)$ denote the number of paths of length $n$ in the graph $\G_A$, given by, 
	\[
	\tau(n-p+1)=\sum\limits_{w\in\mathcal{L}_n}k(X_1*X_2)\dots k(X_{n-p+1}*X_{n-p+2}).
	\] 
	Set $M:=\max\limits_{Y\in\mathcal{L}_{p}}\left\{\frac{m(Y)}{k(Y)}\right\}>0$. Then
	\[
	|\Lambda_n|=\sum_{w\in\mathcal{L}_n}m(w)\le M\sum_{w\in\mathcal{L}_n}k(X_1*X_2)\dots k(X_{n-p+1}*X_{n-p+2}) = M \, \tau(n-p+1).
	\]
Also by part (i), $\tau(n-p+1)\le |\Lambda_n|$, for each $n>p$. Hence,  
\[
h_{top}(\Sigma_A)=\lim_{n\rightarrow \infty} \dfrac{1}{n}\ln \tau(n)=\lim_{n\rightarrow \infty} \dfrac{1}{n}\ln |\Lambda_n|.
\]
(iii) Finally, let $r\in\R$ be a repeated word with multiplicity $m\ge 2$. Since $\Sigma_\F$ is irreducible, there exist a word $s\in\Lambda_p$ that contains $r$ and a word $sws\in\Lambda^p$. Consider the subsequence $(n_j)_{j\ge 1}$, where $n_j=j(|s|+|w|)$. The word $r$ appears in the word $swsw\dots sw\in\Lambda_{n_j}$ at least $j$ many times. 
Thus $|\Lambda_{n_j}|\ge m^j$ and hence 
\[
h_{\text{top}}(\Sigma_A) = \lim_{j\rightarrow\infty}  \dfrac{1}{n_j}\ln |\Lambda_{n_j}| \ge  \lim_{j\rightarrow\infty}  \frac{j\ln(m)}{j\left(|s|+|w|\right)}=\frac{\ln(m)}{|s|+|w|}>0.
\]
\end{proof}
\begin{definition}[Adjacency matrix associated with the collections $\F$ and $\R$]
The matrix $A$, defined by Equation~\eqref{A_for_F_R} is called the \emph{adjacency matrix associated with the collections $\F$ and $\R$}.
\end{definition}
\begin{exam}\label{exam:1}
Let $\Sigma=\{0,1\}$, $\F=\{010,101,111\}$, and $\R=\{00(2),0110(3)\}$. Here $p=4$ and $\Lambda_{3}=\{(000,4),(001,2),(011,1),(100,2),(110,1)\}$. 
The graph corresponding to the adjacency matrix $A$ associated with the collections $\F$ and $\R$ is shown in Figure~\ref{fig:G_A}.
\end{exam}
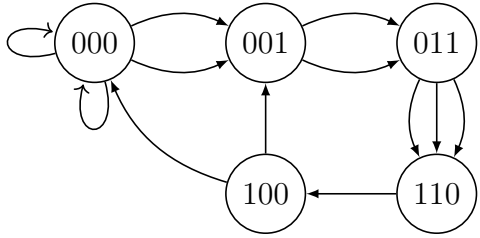
\begin{figure}[h]
\centering
\begin {tikzpicture}[-latex, auto ,node distance =2cm and 2.25cm, ,on grid, semithick, state/.style ={draw, circle}]
\node[state] (A) {$000$};
\node[state] (B) [right =of A] {$001$};
\node[state] (C) [right =of B] {$011$};
\node[state] (D) [below =of B] {$100$};
\node[state] (E) [right =of D] {$110$};
\path (A) edge [loop below] node[above] {} (A);
\path (A) edge [loop left] node[above] {} (A);
\path (A) edge [bend left =25] node[above] {} (B);
\path (A) edge [bend right =25] node[below] {} (B);
\path (B) edge [bend left =25] node[above] {} (C);
\path (B) edge [bend right =25] node[below] {} (C);
\path (C) edge [bend left =25] node[above] {} (E);
\path (C) edge node {} (E);
\path (C) edge [bend right =25] node[above] {} (E);
\path (D) edge [bend left =25] node[below] {} (A);
\path (E) edge node {} (D);
\path (D) edge node {} (B);
\end{tikzpicture}
\caption{Graph $\G_A$ for Example~\ref{exam:1}}
\label{fig:G_A}
\end{figure}
\begin{remark} A few observations about the matrix $A$:\\
(1) Since $\R$ is reduced, for $X*Y\in\mathcal{L}_p$,
\begin{equation}\label{eq:ob1}
	k(X*Y)=	\begin{cases}
		k(X),\quad &\text{if } X*Y\notin\R\\
		m(X*Y),\quad & \text{otherwise}.
	\end{cases}
\end{equation}	
(2) The non-zero entries of the row of $A$ indexed by a repeated word $X$ (with $m(X)>1$) are all equal to $k(X)$, using Equation~\eqref{eq:ob1}. \\
(3) Even though the Perron root of $A$ is given by $\lim_{n \to \infty}|\Lambda_n|^{1/n}$, note that $|\Lambda_n|$ may not be equal to the sum of entries of $A^{n-p+1}$. This equality holds when words in $\R$ are all of equal length $p$, by Theorem \eqref{thm:lang}. \\
(4) One can define a new graph by the outsplitting of $A$ at each vertices with multiplicity at least 2 (we refer to~\cite{LM_book} for the definition of outsplitting) and then adding a few stranded vertices (i.e., with no outgoing edges) so that the total number of paths on this new graph is exactly equal to as in $\Lambda^p$. The shift on this graph and $\Sigma_A$ are conjugates, since the stranded vertices do not contribute to the infinite sequences. Hence we can prove the existance of a graph where finite paths completely describle $\Lambda^p$.
\end{remark}

\begin{exam}\label{exam:diff}
Let us review the example given in Remark~\ref{rem:full_matrix}. Consider $\Sigma=\{0,1\}, \ \F=\{00\}$ and $\R=\{110(2),01(3)\}$. Here $p=3$ and the adjacency matrix $A$ (indexed by $\{01,10,11\}$) associated with $\F$ and $\R$ is given by $\begin{pmatrix}
0&3&3\\1&0&0\\0&2&1
\end{pmatrix}$, and its associated graph is shown in Figure~\ref{fig:Graph_for_F_R}.
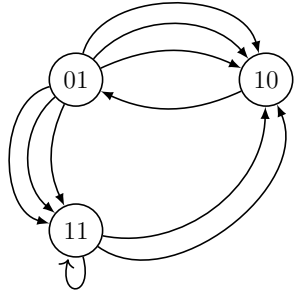
\begin{figure}[h]
\centering
$ \begin{tikzpicture}[-latex, baseline=-1.2cm, auto=left, node distance 					=2cm and 2.5cm ,on grid, semithick, state/.style ={draw, circle}]
\node[state,scale=0.8] (A) {$01$};
\node[state,scale=0.8] (B) [right =of A] {$10$};
\node[state,scale=0.8] (C) [below =of A] {$11$};
\path (C) edge [loop below] node[above] {} (C);
\path (A) edge [bend left =25] node[above] {} (B);
\path (A) edge [bend left =50] node[above] {} (B);
\path (A) edge [bend left =75] node[above] {} (B);
\path (A) edge [bend right =25] node[left] {} (C);
\path (A) edge [bend right =50] node[left] {} (C);
\path (A) edge [bend right =75] node[left] {} (C);
\path (B) edge [bend right =-25] node[below] {} (A);
\path (C) edge [bend left =-50] node[above] {} (B);
\path (C) edge [bend left =-75] node[above] {} (B);
\end{tikzpicture}
$
\caption{Graph $\G_A$ for Example~\ref{exam:diff}}
\label{fig:Graph_for_F_R}
\end{figure}
Note that $h_{\text{top}}(\Sigma_A)=\lim_{n \to \infty}\frac{1}{n}\ln(|\Lambda_n|)\sim\ln(2.6)$. Here sum of entries of $A$ is 9 but $|\Lambda_3|=|\{(010,3),(011,3),(101,3),(110,2),(111,1)\}|=12$.
\end{exam}

\noindent Consider the following matrix $\tilde{A}$, indexed by $\mathcal{L}_{p-1}$, given by 
\[\tilde{A}_{XY} = 
\begin{cases}
m(X*Y), & X*Y\in\mathcal{L}_p \\
0, & \text{ otherwise}.
\end{cases}
\]	
Clearly $\mathcal{L}(\Sigma_{\tilde{A}}) \supseteq \Lambda^p$, where $\mathcal{L}(\Sigma_{\tilde{A}})$ denotes the language of $\Sigma_{\tilde{A}}$. When all the words in $\R$ have length $p$, then the adjacency matrix associated with $\F$ and $\R$ equals $\tilde{A}$. However, as in Example~\ref{exam:diff}, the matrices $A$ and $\tilde{A}$ can be different.

Let $\mathcal{L}(S)$ denote the language of an SFT $S$ with symbols from $\Sigma$. We will visualize $S$ using its $p$-block presentation. Let $\mathcal{S}$ be the collection of all subshifts of finite type $S$ such that $\mathcal{L}(S)\subset \Lambda^{p}$.  Clearly $\Sigma_A\in\mathcal{S}$ where $A$ is the adjacency matrix associated with the collections $\F$ and $\R$.
Let $S\in\mathcal{S}$. If $\tau_S(n)$ denotes the number of allowed words of length $n$ in $S$, then $\tau_S(n)\le |\Lambda_n|$. Hence $h_{top}(S)\le \ln \theta$. Moreover, we have the following result.


\begin{theorem}
There is no subshift of finite type $S\in\mathcal{S}$ that satisfies $\mathcal{L}(\Sigma_A)\subsetneq  \mathcal{L}(S)\subsetneq\Lambda^p.  $
\end{theorem}
\begin{proof}
Suppose there is such a subshift $S$. Let $t>1$ be the length of the shortest word in $\mathcal{L}(S)\setminus\mathcal{L}(\Sigma_A)$. Consider the adjacency matrices $A',A_S'$ corresponding to the $(t-1)$-step shift of both $\Sigma_A$ and $S$ respectively. Note that since $A'<A_S'$, their respective Perron roots satisfy $\theta<\lambda_S$. Hence $\ln \theta \lneq h_{top}(S)$ which is a contradiction since $\mathcal{L}(S)\subsetneq\Lambda^p$. 
\end{proof}

\begin{remark}\label{rem:new_R}
If the collection $\R$ has words of length strictly smaller than $p$, then we define a new extended collection
\[
\tilde{\R}=\{X*Y\ :\ X,Y\in\Lambda_{p-1},\ X*Y\in\Lambda_p \text{ and begins with a word from $\R$}\}.
\]
For $X*Y \in \tilde{\R}$, define its multiplicity to be $k(X*Y)$. Clearly, the adjacency matrix associated with $\F$ and $\tilde{\R}$ is the same as the adjacency matrix $A$ associated with $\F$ and $\R$. Since all the words in $\tilde{\R}$ are of length $p$, by Theorem \ref{thm:lang}, $\mathcal{L}(\Sigma_A)$ can be characterized using the generalized language associated with this new collection $\F$ and $\tilde{\R}$.
\end{remark}

\medskip

\noindent
For the rest of the paper, let us fix $q \ge 2$, a finite symbol set $\Sigma$ of size $q$ and reduced collections $\mathcal{F} = \{a_{1},\dots, a_{s}\}$ of forbidden words and $\mathcal{R} = \{ r_{1}(m_{1}), \dots, r_{\ell}(m_{\ell}) \}$ of repeated words with symbols from $\Sigma$. Here $m_{j} > 1$ for all $1 \le j \le \ell$. Fix $p$ to denote the length of the longest word in $\F \cup \R$ and $\Lambda^p = \cup_{n\ge p} \lambda_{n}$ be the generalized language associated with $\F$ and $\R$.  

\section{Enumeration of words in the generalized language}\label{sec:recurrence relations}
In this section, we study the asymptotic behaviour of $|\Lambda_{n}|$ (see Definition~\ref{def:lambda_n}) for given $\F$ and $\R$, using combinatorial tools. This gives the entropy of the subshift $\Sigma_A$ obtained in Theorem \ref{thm:same_roots}. We consider two cases depending on whether $\F \cup \R$ is reduced or not. 

\noindent
Let $u$ and $v$ be two words with symbols from $\Sigma$ and let $(u, v)=(c_{1},\dots,\,c_{|u|})$ be the correlation between $u$ and $v$ (see Definition~\ref{def:corr}). 

\begin{notation}
We say that $t \in (u,v)$ for $t > 0$, if the $t^{th}$ element in $(u,v)$ counted from the right is 1, that is, $c_{|u|-t+1} =1$. Thus the correlation polynomial can be written as
\[
(u,v)_{z} \  = \sum\limits_{0 \,< \,t \,\in \,(u,v)} z^{t-1}.
\]  
\end{notation}
Let $\mathcal{L}_n$ be the set of all allowed words of length $n$ in $\Sigma_\F$ and let $w = x_{1}x_{2} \dots x_{n}\in\mathcal{L}_n$. Observe that if one of the given repeated words $r_{j}$ is a subword of $w$, then there exists $t \in (w, r_{j})$ such that $t \ge |r_{j}|$. Let $f(n)=\sum_{w\in\mathcal{L}_n}m(w) = |\Lambda_n|$ (see Definition~\ref{def:lambda_n}). By convention, let $f(0) = 1$. 

For each $1\le j\le \ell$, let $g_{r_{j}}(n)=\sum_wm(w)$, where summation is over all the words $w \in \mathcal{L}_{n}$ which end with a repeated word $r_{j} \in \mathcal{R}$. Note that such a word $w$ can have more repeated words as subwords. By convention, let $g_{r_j}(0) = 0$. 

Although the multiplicity of any forbidden word is $0$ by Definition~\ref{def:multi}, we next redefine the multiplicity of certain kinds of forbidden words $w$, only limited to this section. For convenience of notation, we denote it also as $m(w)$.

\begin{definition}\label{def:multi_forb}
[Multiplicity of a word which contains a forbidden word as a subword only at the end -- \textit{only limited to the discussion in Section~\ref{sec:recurrence relations}}] Let $w$ be a forbidden word in $\Sigma_\F$ such that $w$ ends with $a_i \in \F$ and this $a_i$ at the end is the only occurrence of a forbidden word in $w$. We define multiplicity of $w$, denoted as $m(w)$, by
\[m(w) := \frac{\prod_{r_j\in\R} m_j^{n(w;r_j)}}{\prod_{r_j\in\R} m_j^{n(a_i;r_j)}}, \]
where $n(u;v)$ denotes the number of times the word $v$ appears as a subword of the word $u$. In particular, for each $1\le i\le s$, $m(a_i)=1$.
\end{definition}

For $1 \le i \le s$, let $f_{a_{i}}(n)=\sum_w m(w)$, where the summation is over all the words $w$ of length $n$ with symbols from $\Sigma$, which end with a forbidden word $a_{i} \in \mathcal{F}$, and moreover, this $a_i$ is the only occurrence of a forbidden word in $w$. By convention, $f_{a_i}(0) = 0$.

The corresponding generating functions for $f(n)$, $g_{r_{i}}(n)$ and $f_{a_{i}}(n)$ are defined as follows.
\[ F(z) \ = \sum\limits_{n \ge 0} f(n) z^{-n} , \quad G_{r_{j}}(z) \ = \sum\limits_{n \ge 0} g_{r_{j}}(n) z^{-n} , \ \ 1 \le j \le \ell,\]
\[ F_{a_{i}}(z) \ = \sum\limits_{n \ge 0} f_{a_{i}}(n) z^{-n} , \ \ 1 \le i \le s. \]

When $\R$ is an empty set, the generating functions $F$ and $F_{a_i}$ are described in Theorem~\ref{thm:Guibas}. In the general case, the generating functions $F, F_{a_i}, G_{r_j}$ will be described in Theorem~\ref{thm:reduced_gen_fun}, when $\F\cup\R$ is a reduced collection, and in Theorem~\ref{thm:non_reduced_gen_fun}, when $\F\cup\R$ is not a reduced collection.\\
In the following remark, we set up new terminologies which we will use in this section, and also make some related observations.  

\begin{remark}
Let $w\in\Lambda_n$.
\begin{enumerate}
\item By definition of $f(n)$, we say that $w$ gets counted in $f(n)$ or simply, $w$ is in $f(n)$.
\item We say that $w $ gets counted in $f_{a_{i}}(n+t)$ (or $g_{r_{j}}(n+s)$) for $t \le |a_{i}|$ (or $s \le |r_{j}|$) if there exists a word $w'$ in $f_{a_{i}}(n+t)$ (or $g_{r_{j}}(n+s)$) which begins with $w$. Note that this $w'$ is unique for a given $w$ when considered without the multiplicities.
\item For $w$ and $w'$ as above, we say that $w$ gets counted $m(w')$ times in $f_{a_{i}}(n+t)$ (or $g_{r_{j}}(n+s)$). Note that $m(w') \ge m(w)$, and moreover if a repeated word appears at some position in $w'$ and not as a subword of $w$, then $m(w') > m(w)$.
%
\item A repeated word $r_{j}$ in $w'$ gives $m_{j}-1$ many extra counting of $w$ in $f_{a_{i}}(n+t)$ (or $g_{r_{j}}(n+s)$) if $r_{j}$ appears in $w'$ other than being a subword of $w$. 
\item 	For any two words $u$ and $v$, let $uv$ denote the word obtained by concatenation of $v$ as a suffix to $u$. We say that a word $w$ appears on the join of $u$ and $v$ if there exists $t>0$ such that $0 < |w| -t \in (u,w)$ and $t \in (w, v)$. That is, some initial part of $w$ overlaps with the end part of $u$ and the remaining part of $w$ overlaps with the beginning part of $v$.
\end{enumerate}
\end{remark}

Suppose now that the collection $\mathcal{R}$ of words with repetitions is non-empty. Consider the collection $\F\cup\R$. In Subsection~\ref{subsec:case1_reduced}, we will discuss the first case where $\F\cup\R$ is reduced. In this case, any forbidden word from $\F$ does not contain a word from $\mathcal{R}$ as its subword. The situation gets complex in general, when $\F\cup\R$ is not reduced, that is, some repeated word sits inside a forbidden word as a subword. This will be discussed in Subsection~\ref{subsec:case2_non_reduced}.

\subsection{When $\F\cup\R$ is reduced} \label{subsec:case1_reduced}

In this case, the collections $\F \ \text{and } \R$ will be automatically reduced and no repeated word in $\R$ appears as a subword of a word in $\F$. 

\begin{theorem}\label{thm:reduced_gen_fun}
Let $\mathcal{F} = \{a_{1},\dots, a_{s}\}$ and $\mathcal{R} = \{ r_{1}(m_{1}), \dots, r_{\ell}(m_{\ell}) \}$ be the collections of forbidden and repeated words respectively, where $m_{j} > 1$ for all $1 \le j \le \ell$. If $\F\cup\R$ is reduced, then the generating functions $F(z), G_{r_{j}}(z)$ and $F_{a_{i}}(z)$ satisfy the following system of linear equations.
\begin{eqnarray}
(z-q) F(z)  - \sum\limits_{j = 1}^{\ell} z\left( 1 - \frac{1}{m_{j}} \right)G_{r_{j}}(z) + \sum\limits_{i = 1}^{s} z F_{a_{i}}(z) & =  & z, 	\label{eq:reduced1}\\
F(z)  +   \sum\limits_{j=1}^{\ell} \left[ z \left( 1 - \frac{1}{m_{j}} \right) (r_{j}, r_{k})_{z} - z^{|r_{j}|} \delta_{jk}  \right] G_{r_{j}}(z) - \sum\limits_{i = 1}^{s} z (a_{i}, r_{k})_{z} F_{a_{i}}(z) & = & 0,  \label{eq:reduced2}  \\
&& \text{for } 1 \le k \le \ell, 			\nonumber  \\
F(z) +  \sum\limits_{j = 1}^{\ell} \left[ z \left(1 - \frac{1}			{m_{j}} \right) (r_{j}, a_{k})_{z}   \right] G_{r_{j}}(z) - \ \sum\limits_{i = 1}^{s} z (a_{i}, a_{k})_{z} F_{a_{i}}(z) \ & = 				& \ 0, \ \label{eq:reduced3}\\
&& \text{for } 1 \le k \le s,				\nonumber
\end{eqnarray}
where $\delta_{jk} = 1$, if $j=k$, otherwise $\delta_{jk} = 0$.
\end{theorem}

\begin{proof}
In order to obtain the first equation in the above system, let the word $w$ be counted in $f(n)$ and $x \in \Sigma$ be a symbol. Consider the concatenation of $w$ and $x$ given by $wx$. Then the word $wx$ is counted with multiplicities in either $f(n+1)$, or in $f_{a_{i}}(n+1)$, for some $a_{i} \in \F$. Therefore
\begin{equation}
\label{eq:step1}
\sum_{x \in \Sigma} \sum_{w\in\mathcal{L}_n}m(wx) =  f(n+1) + \sum\limits_{i = 1}^{s} f_{a_{i}}(n			+1).
\end{equation}

To determine the quantity on the left hand side, observe that for each $x \in \Sigma$, the word $wx$ is counted at least $m(w)$ times, which gives total of $q f(n)$ counting. Moreover, only those $wx$ which end with some repeated word $r_{j} \in \mathcal{R}$, are counted $m_{j} - 1$ times extra (one counting is already incorporated in $qf(n)$). The word $wx$ ends with a repeated word if and only if the last symbol of $r_j$ is $x$ and a tail of $w$ matches with the beginning subword of $r_j$ of length $|r_{j}| -1$, that is, $|r_{j}| -1 \in (w, r_{j})$. Hence
\[ 
m_j\sum\limits_{\{ w\in\mathcal{L}_n\ :\ |r_j|-1\in(w,r_j)\}}m(w)= g_{r_{j}}(n+1).			\]

Since each $w$ satisfying $|r_{j}|-1 \in (w, r_{j})$ gets counted extra $m_{j}-1$ times when concatenated with an $x$ same as the last symbol of $r_{j}$ at the end, we have
\begin{equation}
\label{eq:step1_LHS}
\sum_{x \in \Sigma} \sum\limits_{w\in\mathcal{L}_n}m(wx) = \ q  f(n)  +  \sum\limits_{j=1}^{\ell}  \frac{g_{r_{j}}(n+1)}{m_{j}} (m_{j}-1). 
\end{equation}
Combining~\eqref{eq:step1} and~\eqref{eq:step1_LHS}, we get
\begin{equation}
\label{eq:step1_final}
q f(n) -  f(n+1) + \sum\limits_{j=1}^{\ell} \left( 1 -\frac{1}{m_{j}} \right) g_{r_{j}}(n+1)  - \sum\limits_{i = 1}^{s} f_{a_{i}}(n+1) \ = \ 0.
\end{equation}
On multiplying both the sides of~\eqref{eq:step1_final} by $z^{-n}$, taking summation over $n \ge 0$, making use of the conventions $f(0) = 1$, $f_{a_{i}}(0) = g_{r_{j}}(0) = 0$, we obtain first equation \eqref{eq:reduced1} in the system.\\~\\
Next, consider an allowed word $w$ counted in $f(n)$, $r_{k} \in \mathcal{R}$, and the word $w r_{k}$ obtained by concatenation. If $w r_{k}$ does not contain any forbidden word, then $w$ is counted in $g_{r_{k}}(n+|r_{k}|)$. Note that it may be counted more than the multiplicity $m(w)$, due to possible appearance of repeated words from $\R$ on the join and at the end. We will take care of these extra counting later. 

Further if some forbidden word from $\F$ appears in $w r_{k}$, then we look at the first occurrence of a forbidden word in $w r_{k}$. Suppose the forbidden word from $\F$ which occurs first is $a_{i}\in\F$. Since $w$ does not contain a forbidden word and since $\F\cup\R$ is reduced, the word $a_{i}$ appears on the join of $w$ and $r_{k}$. Thus there exists $0 < t \in (a_{i}, r_{k})$ such that $w$ is counted in $f_{a_{i}}(n+t)$ (counted more than $m(w)$ times if repeated words appear on the join and before the placement of the word $a_i$).

Now, we need to take away the extra counting of the words as mentioned before. Figure~\eqref{fig:reduced_rep_forb} illustrates situations like these. Suppose some $r_{j} \in \mathcal{R}$ appears on the join of $w$ and $r_{k}$, then there exists $0<s \in (r_{j}, r_{k})$. The number of words $w$ of length $n$ that give $r_{j}$ on the join (and no forbidden word before the placement of $r_j$) is precisely $\frac{g_{r_{j}} (n+s)}{m_{j}}$. But, each such $w$ gets counted $m_{j} - 1$ times more in $g_{r_{j}}(n+s)$, due to the multiplicity of the occurrence of $r_{j}$ at the end. Thus, we subtract such extra counting $\frac{g_{r_{j}} (n+s)}{m_{j}} (m_{j} - 1)$ to obtain
\begin{equation}
\label{eq:step2_final}
f(n)  \ =  \ g_{r_{k}}(n + |r_{k}|) \ + \ \sum\limits_{i=1}^{s} \ \sum\limits_{0 \, < \, t \, \in \, (a_{i}, r_{k})} f_{a_{i}}(n+t)\ - \ \sum\limits_{j = 1}^{\ell} \ \sum\limits_{0 \, < \, s \, \in \, (r_{j}, r_{k})} \left( 1 - \frac{1}{m_{j}} \right) g_{r_{j}} (n+s).
\end{equation}

The condition that $\F\cup\R$ is reduced implies that $t < |a_{i}|$ and $s < |r_{j}|$ in~\eqref{eq:step2_final}. Multiply the above equation by $z^{-n}$ and take sum over $n \ge 0$ to obtain~\eqref{eq:reduced2}.

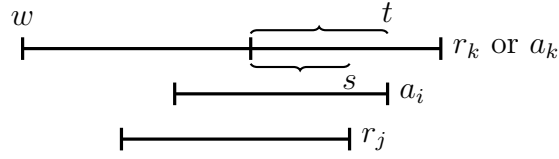
\begin{figure}[h!]
\begin{center}
\begin{tikzpicture}
\beginpgfgraphicnamed{situation-b}
\begin{scope}[very thick]
\draw[snake=brace] (3,0.2) -- (4.8,0.2) [thick] node[above] 						{$t$};
\draw (0,-0.15) -- (0,0.15)  node[anchor=south] {$w$} ;
\draw (0,0) -- (3,0) -- (5.5,0)  node[anchor=west] {$r_{k}$ or $a_{k}$};
\draw (3,-0.15) -- (3,0.15);
\draw (5.5,-0.15) -- (5.5,0.15);
\draw[snake=brace, mirror snake] (3,-0.2) -- (4.3,-0.2) 							[thick] node[below] {$s$};
\draw (2,-0.6) -- (4.8,-0.6) node[anchor=west] {$a_{i}$};
\draw (2,-0.75) -- (2,-0.45);
\draw (4.8,-0.75) -- (4.8,-0.45);
\draw (1.3,-1.2) -- (4.3,-1.2) node[anchor=west] {$r_{j}$};
\draw (1.3,-1.05) -- (1.3,-1.35);
\draw (4.3,-1.05) -- (4.3,-1.35);
\end{scope}
\endpgfgraphicnamed
\end{tikzpicture}
\caption{Words $a_{i}$ and $r_{j}$ appear on the join, with $r_{j}$ placed before $a_{i}$.}
\label{fig:reduced_rep_forb}
\end{center}
\end{figure}
Finally, in order to obtain the last set of equations~\eqref{eq:reduced3}, consider the concatenated word $wa_k$, where $w$ is counted in $f(n)$ and $a_{k} \in \F$. Suppose $a_{i}$ is the first occurrence of a forbidden word in $w a_{k}$. Then such $a_{i}$ appears either on the join of $w$ and $a_{k}$, or the attached $a_{k}$ at the end is the only occurrence of a forbidden word in $w a_{k}$. If $a_{i}$ appears on the join then there exists $0<t \in (a_{i}, a_{k})$ such that $w$ is counted in $f_{a_{i}}(n+t)$. Clearly there will be some extra counting due to possible appearances of repeated words on the join of $wa_k$, same as the previous case, also shown in Figure~\ref{fig:reduced_rep_forb}. Therefore the same counting argument described previously applies here and we obtain,
\begin{equation}
\label{eq:step3_final}
f(n)\ = \ \sum\limits_{i=1}^{s} \ \sum\limits_{0 \, < \,t \, \in \, (a_{i}, 		a_{k})} f_{a_{i}}(n+t)\ - \ \sum\limits_{j = 1}^{\ell} \ \sum\limits_{0 \, < 	\, s \, \in \, (r_{j}, a_{k})} \left( 1 - \frac{1}{m_{j}} \right)g_{r_{j}} 		(n+s).
\end{equation}

To complete the proof, we multiply~\eqref{eq:step3_final} by $z^{-n}$ and take sum over $n \ge 0$ to obtain the last equation~\eqref{eq:reduced3} in the system.
\end{proof}

\begin{exam}\label{exam:reduced}
We illustrate Theorem~\ref{thm:reduced_gen_fun} using this example. For $q = 2$, let $\mathcal{R} = \left\lbrace 000 (2) \right\rbrace$ and $\mathcal{F} = \left\lbrace 010 \right\rbrace$. Clearly $\F\cup\R$ is a reduced collection. Denote $r_1 = 000$ with $m_1=2$ and $a_1 = 010$. Then, $1 \in (a_1,r_1)$, $1,2,3 \in (r_1,r_1)$, $1,3 \in (a_1,a_1)$ and $1 \in (r_1,a_1)$. Equations ~\eqref{eq:step1_final},~\eqref{eq:step2_final} and ~\eqref{eq:step3_final} can be verified using the values given in the table below.

%

%
\begin{center}
\begin{tabular}{ | c | c c c c c c c c  c  c | }
\hline
$n$  & 1 & 2 & 3 & 4 & 5 & 6 & 7 & 8 & 9 & 10 \\
\hline 
$f(n)$ & 2 & 4 & 8 & 17 & 37 & 81 & 178 & 392 & 864 & 1905 \\
\hline
$g_{r_{1}}(n)$ & 0 & 0 & 2 & 6 & 14 & 32 & 72 & 160 & 354 & 782 \\
\hline
$f_{a_{i}}(n)$ & 0 & 0 & 1 & 2 & 4 & 9 & 20 & 44 & 97 & 214 \\
\hline
\end{tabular}
\end{center}	
\end{exam} 

%

\subsubsection*{Matrix representation of the system of linear equations obtained in Theorem~\ref{thm:reduced_gen_fun}}
Consider the matrix function $L(z)$ given as
\begin{equation*}
\label{eq:matrix_L}
L(z) \ = \ \begin{pmatrix}
\begin{matrix}
z-q 
\end{matrix}
& \rvline & \begin{matrix}
-z(1-\frac{1}{m_{1}}) & \dots & -z(1-\frac{1}{m_{\ell}}) & z \mathds{1}_{s}^{T}
\end{matrix} \\
\hline
\begin{matrix}
\mathds{1}_{\ell + s}
\end{matrix} & \rvline & \begin{matrix}
\ & \ & P(z) & \
\end{matrix} \\
\end{pmatrix},
\end{equation*}
where $\mathds{1}_m$ denotes the column matrix of size $m$ with all entries $1$, $P(z)=[P(z)_{ij}]_{1\le i,j\le \ell+s}$ with
\begin{equation}\label{eq:matrix_P}
P(z)_{ij}  = \begin{cases}
z(1 - \frac{1}{m_{j}})(r_{j}, r_{i})_{z} - \delta_{ij} z^{|r_{j}|}, \ \ \ & 1 \le i, j \le \ell, \\
- z(a_{j}, r_{i})_{z}, \ \ \ & 1 \le i \le \ell, \ \ell+1 \le j \le \ell+s,  \\
z(1 - \frac{1}{m_{j}})(r_{j}, a_{i})_{z}, \ \ \ & \ell+1 \le i \le \ell+s, \, 1 \le j \le \ell, \\
- z(a_{j}, a_{i})_{z}, \ \ \ & \ell+1 \le i, j \le \ell+s
\end{cases},
\end{equation}
where $\delta_{ij} = 1$ if $i = j$, else $\delta_{ij} = 0$. \\
Using the matrix function $L(z)$, the linear system of equations in Theorem~\ref{thm:reduced_gen_fun} can be expressed as				\begin{equation*}
\label{eq:matrix_system}
L(z) \begin{pmatrix}
F(z) \\ G_{r_{1}}(z)\\ \vdots \\ G_{r_{\ell}}(z) \\ F_{a_{1}}(z) 				\\ \vdots \\ F_{a_{s}}(z)
\end{pmatrix} \ = \ \begin{pmatrix}
z \\ 0 \\ \vdots \\ 0 \\ 0 \\ \vdots \\ 0
\end{pmatrix}.
\end{equation*}
We make an important observation, that the matrix $L$ is invertible, since the highest degree polynomials occur on diagonal of $L$, giving a non-zero determinant of $L$. 
\begin{theorem}\label{thm:gen1}
With notations as above, the generating function $F(z)$ takes the following form of a rational function:
\begin{equation}
\label{eq:0}
F(z)=\frac{z}{z-q+z\sum_{i=1}^\ell \left(1-\frac{1}{m_i}\right)R_i(z)-z		\sum_{j=1}^sR_{\ell+j}(z)},
\end{equation}
where $R_{i}(z)$ is the $i^{th}$ row sum of $P^{-1}(z)$. 
\end{theorem}

\begin{proof}
The result follows by computing $L(z)^{-1}$ using the formula for inverse of a $2 \times 2$ block matrix given in~\cite{Inverse}.
\end{proof}

We will now obtain another form of $F(z)$ in the following result. We will first define notations required for the result. Consider the diagonal matrix function $D(z)=\text{diag}\{D_{ii}(z)\}_{1\le i\le \ell+s}$, where 
\begin{equation}\label{eqn:diagD}
D_{ii}(z)=\begin{cases}
z\left(1-\frac{1}{m_i}\right), & 1\le i\le\ell,\\ -z, &  \ell+1\le i\le \ell+s.
\end{cases}
\end{equation}
Define another matrix function 
\begin{eqnarray}\label{eq:Q}
Q&=&D^{-1}P^T D.
\end{eqnarray} 

\begin{theorem}\label{thm:gen2} 
With notations as above, $F(z)$ takes the following form of a rational function:
\begin{eqnarray}\label{eq:00}
F(z)&=&\frac{z}{z-q+z\sum_{i=1}^\ell \left(1-\frac{1}{m_i}\right)S_i(z)-z\sum_{j=1}^sS_{\ell+j}(z)},
\end{eqnarray}
where $S_i(z)$ denotes the $i^{th}$ row sum of inverse of $Q^{-1}(z)$.
\end{theorem}

\begin{proof}
The proof follows by observing the relationship between the matrix functions $P$ and $Q$ and Theorem~\ref{thm:gen1}.
\end{proof}

\begin{remark}
If the collection $\mathcal{R}$ is empty, the matrix $L(z)$ is equal to $K(z)$ as defined in Theorem~\ref{thm:Guibas}. Consequently, Theorem~\ref{thm:Guibas} is obtained as a corrollary to the Theorem~\ref{thm:reduced_gen_fun}.
\end{remark}

%

\subsection{When $\F\cup\R$ is not reduced} \label{subsec:case2_non_reduced}

In this section, we consider the reduced collections $\F$ and $\R$ such that $\F\cup\R$ is not reduced. Since every word in $\R$ is an allowed word, it cannot contain a forbidden word as its subword. However, since $\F\cup\R$ is not reduced, some forbidden word in $\F$ contains a word from $\R$ as its subword. That is, in this section, we are allowing for existence of $a_{i} \in \F$, $r_{j} \in \R$ and $t \ge |r_{j}|$ such that $t \in (a_{i}, r_{j})$.

Let $w$ be in $f(n)$ and suppose it gets counted in $f_{a_{i}}(n+t)$ for some $0 < t \le |a_{i}|$. Let $w'$ be the corresponding word in $f_{a_{i}}(n+t)$. Suppose $r_{j}$ is a subword of $a_{i}$ which appears in $w'$ such that this occurrence of $r_{j}$ is not contained in $w$. Then by considering the redefined multiplicity of the forbidden $w'$ given in Definition~\ref{def:multi_forb}, this occurrence of $r_{j}$ does not contribute to any extra counting of $w$ in $f_{a_{i}}(n+t)$. 

We now make another important observation. If $w$ in $f(n)$ gets counted in $f_{a_{i}}(n+t)$ for some $0 < t \le |a_{i}|$ and $a_{i}$ contains some $r_{j}$ as a subword except at the end, then clearly $w$ also gets counted in $g_{r_{j}}(n+s)$ for some $0 < s < t$. However if $r_{j}$ appears as a subword of $a_{i}$ at the end, then $w$ can not get counted simultaneously in both $g_{r_{j}}(n+t)$ and $f_{a_{i}}(n+t)$, as $g_{r_{j}}$ counts only the allowed words ending with $r_{j}$. Therefore, we introduce the following notation.

\begin{notation}
The total number of distinct positions at which $r_{j}$ appears in $a_{i}$ as a subword, except for at the end is denoted as 
\begin{eqnarray}\label{eq:gamma}
\gamma(a_{i}, r_{j}) &=& \# \left\lbrace \alpha \in (a_{i}, r_{j}) : \alpha > |r_{j}| \right\rbrace.
\end{eqnarray} 
\end{notation} 

\noindent We also define another concept which will be useful in this section. 
\begin{definition}[$\alpha$-tail correlation and $\alpha$-tail correlation polynomial]
Let $u$ and $v$ be two words with symbols from $\Sigma$, $(u,v) =(c_{1},\dots,\,c_{|u|})$ be the correlation of $u$ and $v$ and let $ 1 \le \alpha \le |u|$. We define the \textit{$\alpha$-tail correlation of $u$ and $v$}, denoted by $(u,v)^{\alpha}$, as the binary string of length $\alpha$ consisting of the last $\alpha$ elements of $(u,v)$, that is,
\[(u,v)^{\alpha} \ = \ \left(c_{|u|-\alpha+1},\, c_{|u|-\alpha+2}, \, \dots,\,c_{|u|} \right).\]
The $\alpha$-tail correlation of $u$ and $v$ can also be interpreted as a polynomial in some variable $z$ as defined below, and is termed as the \textit{$\alpha$-tail correlation polynomial of $u$ and $v$}. For a given $\alpha \ge 1$, the $\alpha$-tail correlation polynomial of $u$ and $v$, denoted as $(u,v)^{\alpha}_{z}$, is defined as
\[ (u,v)^{\alpha}_{z} \ = \ \sum\limits_{i=\alpha}^{|u|} c_{i} z^{|u|-i} \ = \sum\limits_{0 \,< \,t \,\in \,(u,v)^{\alpha}} z^{t-1}.\]
\end{definition}

\begin{remark}
For $\alpha = |u|$, $(u,v)^{|u|} = (u,v)$ and $(u,v)^{|u|}_z = (u,v)_z$.
\end{remark}

\begin{exam}

Let $u = 210210$ and $v = 2102$. Then $(u,v) = \left( 1,0,0,1,0,0 \right)$ and $(v,u) = \left(1,0,0,1 \right)$. The $4$-tail correlation of $u$ and $v$ is $(u,v)^{4} = \left(0,1,0,0 \right)$ and the $2$-tail correlation of $v$ and $u$ is $(v,u)^{2} = \left(0,1 \right)$. Moreover the corresponding polynomials are
\[ (u,v)_{z} = z^{5}+z^{2}; \ (v,u)_{z} = z^{3}+1; \ (u,v)^{4}_{z} = z^{2}; \ (v,u)^{2}_{z} = 1.\]
\end{exam}

\begin{theorem}
\label{thm:non_reduced_gen_fun}
Let $\F = \{a_{1},\dots, a_{s}\}$ be a reduced collection of forbidden words and $\R = \{ r_{1}(m_{1}), \dots, r_{\ell}(m_{\ell}) \}$ be a reduced collection of repeated words with $m_{i} > 1$, such that the union $\F\cup\R$ is not reduced. Then the generating functions $F(z), G_{r_{j}}(z)$ and $F_{a_{i}}(z)$ satisfy the following system of linear equations:
\begin{align}
(z - q) F(z)-\sum\limits_{j=1}^{\ell} z \left( 1 -\frac{1}{m_{j}} \right) 					G_{r_{j}}(z)    	\qquad 			\qquad 	\qquad 		\qquad	\qquad	\qquad \nonumber \\
+ \sum\limits_{i = 1}^{s} z \left[ 1+  \sum\limits_{j = 		1}^{\ell} (m_{j}-1) \gamma(a_{i}, r_{j}) \right]F_{a_{i}}(z)   \ & = \ z \label{eq:non_reduced1} \\
F(z) + \sum\limits_{j=1}^{\ell} \ \left[ z \left( 1- \frac{1}{m_{j}} \right) (r_{j},r_{k})_{z} - z^{|r_{j}|}\delta_{jk} \right]  G_{r_{j}}(z) \	 	\qquad 	\qquad 	\qquad  \label{eq:non_reduced2} \\
- \ \sum\limits_{i=1}^{s} \  z (a_{i}, r_{k})_{z}^{|r_{k}|} 					\left[ 1 \, + \, \sum\limits_{j = 1}^{\ell} (m_{j} -1) \, 							\gamma(a_{i}, r_{j})  \right] F_{a_{i}}(z)& = \ 0, \ \  1 \le k \le \ell,  \nonumber \\
F(z)	+  \ \sum\limits_{j = 1}^{\ell} \ z \left( 1 - \frac{1}{m_{j}} \right) (r_{j}, a_{k})^{|r_{j}|-1}_{z}  G_{r_{j}} (z) \
\qquad 	\qquad 	\qquad 	\qquad \label{eq:non_reduced3} \\
-   \sum\limits_{i=1}^{s} \left[ z (a_{i}, a_{k})_{z} + \sum\limits_{0 \, < \, t \, \in \, (a_{i}, a_{k})} z^{t} \ \sum\limits_{j = 1}^{\ell}(m_{j} -1) \gamma_{t}(a_{i}, r_{j}) \,  \right] F_{a_{i}}(z)
& = \ 0,\ \  1  \le k \le s.  \nonumber
\end{align}

\end{theorem}

The proof of this result is a counting argument as in the proof of Theorem~\ref{thm:reduced_gen_fun}. The counting is trickier here due to the appearance of repeated word as a subword of a forbidden word since the collection $\F\cup \R$ is not reduced. Due to heavy technical calculations, the proof is given in the Appendix~\ref{sec:app}.

\begin{remark}
When the collection $\F\cup\R$ is reduced, then $\gamma(a_{i}, r_{j}) = 0$ for all $1 \le j \le \ell, \ 1 \le i \le s$. Moreover, the tail-correlation polynomials $(a_{i}, r_{k})_z^{|r_{k}|}$ and $(r_{j}, a_{k})_z^{|r_{j}|-1}$ coincide with the correlation polynomials $(a_{i}, r_{k})_z$ and $(r_{j}, a_{k})_z$. Thus Theorem~\ref{thm:reduced_gen_fun} is a corollary to Theorem~\ref{thm:non_reduced_gen_fun}.  
\end{remark}

\section{Perron Root and Eigenvectors of $A$}\label{sec:Perron EVectors}
In this section, we obtain a description of the Perron root of the adjacency matrix $A$ associated with given $\F$ and $\R$ using the results from the previous section and also derive expressions for the left and right Perron eigenvectors of $A$ in terms of correlation polynomials. When $\Sigma_A$ is irreducible, we apply Theorems~\ref{thm:same_roots},~\ref{thm:gen1} and~\ref{thm:gen2} together to obtain the Perron root of $A$ as the largest real pole of $F(z)$. The proof follows along the same lines as the proof of~\cite[Theorem 4.1]{Parry}. Hence we immediately arrive at the following result. 



\begin{theorem}[The Perron root]\label{thm:Perron_root}
With the notations as above, the Perron root of $A$ is given by the largest real zero of $z-q+z\sum_{i=1}^\ell \left(1-\frac{1}{m_i}\right)R_i(z)-z		\sum_{j=1}^sR_{\ell+j}(z)$, which equals $z-q+z\sum_{i=1}^\ell \left(1-\frac{1}{m_i}\right)S_i(z)-z\sum_{j=1}^sS_{\ell+j}(z)$ (see~\eqref{eq:0} and~\eqref{eq:00} for notations).
\end{theorem}

For a given non-negative irreducible integer matrix $A$ of size $q$ labelled by symbols $0,1,\dots,q-1$ (in order), define $\F=\{xy\ : \  A_{xy}=0\}=\{a_1,\dots,a_s\}$ and $\R=\{xy:A_{xy}>1\}=\{r_1,\dots,r_\ell\}$ with multiplicity $m_i$ of $r_i=xy$ given by $A_{xy}$. Then trivially, $\F\cup\R$ is reduced with the adjacency matrix $A$. Hence we can obtain an expression for the Perron root of $A$ using Theorem~\ref{thm:Perron_root}.

\begin{remark}
The Perron root of the adjacency matrix $A$ associated with $\F$ and $\R$ is given by the largest real pole of the generating function $F(z)$ even when $\F\cup\R$ is not reduced. Hence, one can solve the system of equations~\eqref{eq:non_reduced1}-~\eqref{eq:non_reduced3} to calculate the Perron root of $A$. For instance, consider the example when $q = 2$, $\F = \{ 001 \}$ and $\R= \{00(2)\}$, let $a=001$ and $r=00$. Then the system~\eqref{eq:non_reduced1}-~\eqref{eq:non_reduced3} is given by
\begin{equation*}
\begin{pmatrix}
z-2 & \frac{-z}{2} & 2z \\ 1 & \frac{z-z^{2}}{2} & 0 \\ 1 & \frac{z}{2} & -z^{3} \end{pmatrix}		
\begin{pmatrix}
F(z) \\ G_{r}(z)\\ F_{a}(z) 
\end{pmatrix} \ =  \ \begin{pmatrix}
z \\ 0 \\ 0
\end{pmatrix},
\end{equation*}	 
which gives $F(z)=\frac{z}{z-2}$. The largest pole of $F(z)$ is $2$, which is the Perron root of the associated adjacency matrix $A=\begin{pmatrix}
2 & 0 & 0 & 0 \\ 0 & 0 & 1 & 1 \\ 1 & 1 & 0 & 0 \\ 0 & 0 & 1 & 1 						\end{pmatrix}$, indexed by $\{00,01,10,11\}$ in order.
\end{remark}

Now, we will obtain an expression for the Perron eigenvectors of $A$. Let $p$ be the length of the longest word in $\F\cup\R$. Let us assume, from now on in this section, that all the words in $\R$ are of length $p$. Consequently $\F\cup\R$ is reduced.  In this case, $A_{XY} = m(X*Y)$ if $X*Y \in \mathcal{L}_{p}$, and is 0 otherwise. For a word $X$ of length $p-1$ with symbols from $\Sigma$, define
\begin{eqnarray}
U_X&=&1-\theta\sum_{i=1}^\ell \left(1-\frac{1}{m_i}\right)R_i(\theta)(\tilde{r}_i,X)_\theta+\theta\sum_{j=1}^sR_{\ell+j}(\theta)(\tilde{a}_j,X)_\theta, \label{eq:U}\\
V_X&=&1-\theta\sum_{i=1}^\ell \left(1-\frac{1}{m_i}\right)S_i(\theta)(X,r_i)_\theta+\theta\sum_{j=1}^sS_{\ell+j}(\theta)(X,a_j)_\theta, \label{eq:V}
\end{eqnarray}
where $\tilde{r}_i,\tilde{a}_j$ are the words obtained by removing the first symbol of $r_i,a_j$ respectively (note that $(\tilde{r}_i,X)_z=(r_i,X)^{|r_i|-1}_z$, $(\tilde{a}_j,X)_z=(a_j,X)^{|a_j|-1}_z$ are the $|r_i|-1$ and $|a_j|-1$ tail correlations as described in Section~\ref{subsec:case2_non_reduced} respectively) and $(u,w)_\theta$ is the correlation polynomial $(u,w)_z$ evaluated at $z=\theta$. Then we have the following result. As argued in~\cite[Lemma 5.3]{Parry} it is easy to show that $R_i(\theta)$ and $S_i(\theta)$ exist.

\begin{theorem}[Perron eigenvectors]\label{thm:perronvector}
Let $\theta$ be the Perron eigenvalue of the adjacency matrix $A$ associated with $\F$ and $\R$, $U=(U_X)_{X \in \mathcal{L}_{p-1}}$ and $V=(V_X)_{X \in \mathcal{L}_{p-1}}$ be as defined in~\eqref{eq:U} and~\eqref{eq:V}. Then $U$ and $V$ are respectively left and right Perron eigenvectors of $A$.
\end{theorem}

\begin{proof}
We first prove that $V$ is a right eigenvector of $A$ with respect to the Perron root $\theta$. For a fixed allowed word $X=x_1x_2\dots x_{p-1}$ of length $p-1$, we need to show that 
\[
\sum_{b=0}^{q-1}m_{Xb}V_{\tilde{X}b}=\theta V_X,
\] where $\tilde{X}b=x_2x_3\dots x_{p-1}b$ and $m_{Xb}$ is the multiplicity of $Xb=x_1x_2\dots x_{p-1}b$. By convention, if $Xb$ is forbidden, then $m_{Xb}=0$. 
That is, we need to show that,
\begin{align}\label{eq:1}
0=&\sum_{b=0}^{q-1}m_{Xb}\left(1-\theta\sum_{i=1}^\ell \left(1-\frac{1}{m_i}\right)S_i(\theta)(\tilde{X}b,r_i)_\theta+\theta\sum_{j=1}^sS_{\ell+j}(\theta)(\tilde{X}b,a_j)_\theta\right)\nonumber	\\
&-\theta\left(1-\theta\sum_{i=1}^\ell \left(1-\frac{1}{m_i}\right)S_i(\theta)(X,r_i)_\theta+\theta\sum_{j=1}^sS_{\ell+j}(\theta)(X,a_j)_\theta\right)\nonumber\\
=&\sum_{b=0}^{q-1}m_{Xb}-\theta-\theta\left(\sum_{i=1}^\ell \left(1-\frac{1}{m_i}\right)S_i(\theta)\left(\sum_{b=0}^{q-1}m_{Xb}(\tilde{X}b,r_i)_\theta-\theta(X,r_i)_\theta\right)\right)\nonumber\\
&+\theta\left(\sum_{j=1}^s S_{\ell+j}(\theta)\left(\sum_{b=0}^{q-1}m_{Xb}(\tilde{X}b,a_j)_\theta-\theta(X,a_j)_\theta\right)\right).
\end{align}
First of all, we make a couple of observations. Fix $w\in\F\cup\R$. We consider the following two cases:\\
\noindent i) Let $X$ be such that $Xb$ does not end with $w$ for any $b\in\Sigma$. Then for any $z$, 
\begin{equation}\label{eq:2}
\sum_{b=0}^{q-1}(\tilde{X}b,w)_z=z(X,w)_z+1.
\end{equation}

\noindent ii) Let $X$ be such that $Xb_0$ ends with $w$ for some $b_0\in\Sigma$. Such $b_0$ is unique and for any $z$, 
\begin{equation}\label{eq:3}
\sum_{b=0}^{q-1}(\tilde{X}b,w)_z+z^{|w|-1}=z(X,w)_z+1.
\end{equation}

\noindent Let $X$ be such that $Xb_1 =r_1, \dots, Xb_d = r_d$ and $Xb_{d+1},\dots,Xb_{d+n}$ end with $a_1,\dots,a_n$, respectively, for some $0\le d\le \ell,0\le n\le s$ so that $m_1=m_{Xb_1},\dots,m_d=m_{Xb_d}>1, m_{Xb_{d+1}}=\dots=m_{Xb_{d+n}}=0$ and $m_{Xb_{d+n+1}}=\dots=m_{Xb_q}=1$ (rename words from $\F$ or $\R$ if needed). We look at terms in~\eqref{eq:1} separately.  

\noindent For $i=1,\dots,d$ and $k=1,\dots,d$, since $Xb_k=r_k$, we have for $k\ne i, (\tilde{X}{b_k},r_i)_\theta=(r_k,r_i)_\theta$, and for $k=i, (\tilde{X}{b_i},r_i)_\theta=(r_i,r_i)_\theta-\theta^{p-1}$. Also for $k=1,\dots,n$, since $Xb_{d+k}$ ends with $a_k$ and $\F\cup\R$ is reduced, we have $(\tilde{X}{b_{d+k}},r_i)_\theta=(a_k,r_i)_\theta$. Hence using~\eqref{eq:3},

\begin{equation}
\begin{split}
\sum_{b=0}^{q-1}m_{Xb}(\tilde{X}b,r_i)_\theta-\theta(X,r_i)_\theta
=& \sum_{k=1}^d m_k(\tilde{X}{b_k},r_i)_\theta+\sum_{k=d+n+1}^{q}(\tilde{X}{b_k},r_i)_\theta-\theta(X,r_i)_\theta\\
=& \sum_{b=0}^{q-1}(\tilde{X}b,r_i)_\theta-\theta(X,r_i)_\theta+\sum_{k=1}^d(m_k-1)(\tilde{X}{b_k},r_i)_\theta \\
& \qquad  -\sum_{k=1}^n(\tilde{X}{b_{d+k}},r_i)_\theta \\
=&1-m_i\theta^{p-1}+\sum_{k=1}^d(m_k-1)(r_k,r_i)_\theta-\sum_{k=1}^n(a_k,r_i)_\theta.
\end{split}\label{eq:4}
\end{equation}

\noindent Similarly for $i=d+1,\dots,\ell$, using~\eqref{eq:2}, we get
\begin{equation}
\begin{split}
\sum_{b=0}^{q-1}m_{Xb}(\tilde{X}b,r_i)_\theta-\theta(X,r_i)_\theta
=&1+\sum_{k=1}^d(m_k-1)(r_k,r_i)_\theta-\sum_{k=1}^n(a_k,r_i)_\theta.
\end{split}\label{eq:5}
\end{equation}

\noindent Combining~\eqref{eq:4} and~\eqref{eq:5}, we obtain
\begin{equation}
\begin{split}
\sum_{i=1}^\ell&\left(1-\frac{1}{m_i}\right) S_i(\theta)\left(\sum_{b=0}^{q-1}m_{Xb}(\tilde{X}b,r_i)_\theta-\theta(X,r_i)_\theta\right)\\=& \sum_{i=1}^\ell\left(1-\frac{1}{m_i}\right) S_i(\theta)\left(1+\sum_{k=1}^d(m_k-1)(r_k,r_i)_\theta-\sum_{k=1}^n(a_k,r_i)_\theta\right) -\sum_{i=1}^d(m_i-1)\theta^{p-1}S_i(\theta).
\end{split}\label{eq:6}
\end{equation}

\noindent We use similar steps to obtain,
\begin{align}\label{eq:7}
\sum_{j=1}^sS_{\ell+j} &(\theta)\left(\sum_{b=0}^{q-1}m_{Xb}(\tilde{X}b,a_j)_\theta-\theta(X,a_j)_\theta\right) \nonumber\\
&= \sum_{j=1}^sS_{\ell+j}(\theta)\left(1-\sum_{k=1}^d(m_k-1)(r_k,a_j)_\theta-\sum_{k=1}^n(a_k,a_j)_\theta\right).
\end{align}

\noindent Also note that $\sum_{b=0}^{q-1}m_{Xb}=q+\sum_{k=1}^d(m_k-1)-n$. 
Combining~\eqref{eq:6} and~\eqref{eq:7} in Equation~\eqref{eq:1}, we obtain the expression

\begin{equation*}
\begin{split}
q-\theta+&\sum_{k=1}^d(m_k-1)-n - \theta\left( \sum_{i=1}^\ell\left(1-\frac{1}{m_i}\right) S_i(\theta)\left(1+\sum_{k=1}^d(m_k-1)(r_k,r_i)_\theta-\sum_{k=1}^n(a_k,r_i)_\theta\right)\right)\\
+&\theta^p\sum_{i=1}^d(m_i-1)S_i(\theta)
+\theta
\sum_{j=1}^sS_{\ell+j}(\theta)\left(1-\sum_{k=1}^d(m_k-1)(r_k,a_j)_\theta-\sum_{k=1}^n(a_k,a_j)_\theta\right).
\end{split}
\end{equation*}

\noindent We use that $q-\theta-\theta\sum_{i=1}^\ell\left(1-\frac{1}{m_i}\right)S_i(\theta)+\theta\sum_{j=1}^sS_{\ell+j}(\theta)=0$. Now consider the following terms separately:

\begin{equation}\label{eq:8}
\sum_{k=1}^d(m_k-1)\left(1-\theta\left(\sum_{i=1}^\ell\left(1-\frac{1}{m_i}\right)(r_k,r_i)_\theta-\theta^{p-1}\right)S_i(\theta)+\theta\sum_{j=1}^s(r_k,a_j)_\theta S_{\ell+j}(\theta)\right),
\end{equation}
and

\begin{equation}\label{eq:9}
\sum_{k=1}^n\left(\theta\sum_{i=1}^\ell\left(1-\frac{1}{m_i}\right)(a_k,r_i)_\theta S_i(\theta)-\theta\sum_{j=1}^s(a_k,a_j)_\theta S_{\ell+j}(\theta)\right)-n.
\end{equation}

\noindent If $Q_{i,j}(z)$ denotes the $i,j$-th entry of $Q(z)$, then $\sum_{j=1}^{\ell+s}Q_{i,j}(z)S_j(z)=1$, for all $i=1,\dots,\ell+s$ and for all $z$. Using this, we get that both the terms~\eqref{eq:8} and~\eqref{eq:9} are zero.\\

To prove that $U$ is a left eigenvector of $A$ with respect to $\theta$, we define another adjacency matrix $B$ corresponding to the collections $\hat{\F}=\{\hat{a}_1,\dots,\hat{a}_s\}$ and $\hat{\R}=\{\hat{r}_1(m_1),\dots,\hat{r}_\ell(m_{\ell})\}$ where $\hat{w}$ is the reverse of $w$. Unlike $A$, here the rows and columns are indexed by the reverse of the allowed words of length $p-1$. For $X=x_1\dots x_{p-1}$ and $Y=y_1\dots y_{p-1}$, the $\hat{X}\hat{Y}$-th entry of $B$ is given by multiplicity of $\hat{X}*\hat{Y}$ (with respect to $\hat{\F}$ and $\hat{\R}$) where $\hat{X}*\hat{Y}$ is defined if $x_{p-2}\dots x_1=y_{p-1}\dots y_2$ and $\hat{X}*\hat{Y}=x_{p-1}y_{p-1}\dots y_1$. Note that this is same as the $YX$-th entry of $A$. Hence $B$ is the transpose of $A$. Since $\F\cup\R$ is reduced, for any $u,w\in\F\cup\R$, note that $(\hat{u},\hat{w})_z=(w,u)_z$ for all $z$. Hence by the first part of this proof, $B$ has a right Perron eigenvector ($A$ has a left Perron eigenvector) with the $\hat{X}$-th entry ($X$-th entry) given by
\[
1-\theta\sum_{i=1}^\ell \left(1-\frac{1}{m_i}\right)R_i(\theta)(\hat{X},\hat{r}_i)_\theta+\theta\sum_{j=1}^sR_{\ell+j}(\theta)(\hat{X},\hat{a}_j)_\theta.
\]    
To obtain the final expression, observe that $(\hat{X},\hat{r}_i)_z=(\tilde{r}_i,X)_z$ and $(\hat{X},\hat{a}_j)_z=(\tilde{a}_j,X)_z$.
\end{proof}

\begin{exam}\label{ex:evector}
Let $\Sigma=\{0,1\},\ \F=\{010\},\ \R=\{100(3)\}$. Then the adjacency matrix is indexed by $\mathcal{L}_2=\{00,01,10,11\}$ and is given by \[
A= \begin{pmatrix}
1& 1& 0& 0\\
0& 0& 0& 1\\
3& 1& 0& 0\\
0& 0& 1& 1
\end{pmatrix}.\] Here $P(z)=\begin{pmatrix}
-z^3/3& -z^2\\
2 z/3& -z (z^2 + 1)
\end{pmatrix}$, $Q(z)=\begin{pmatrix}
-z^3/3& -z\\
2 z^2/3& -z (z^2 + 1)\\
\end{pmatrix},$ $R_1(z)= \dfrac{-3(1 - z + z^2)}{z^2 (2 + z + z^3)}$, $R_2(z)= \dfrac{-(2 + z^2)}{z^2 (2 + z + z^3)}, S_1(z)= \dfrac{-3}{2 + z + z^3}$, and $S_2(z)= \dfrac{-(2 + z)}{2 z + z^2 + z^4}$. This gives the Perron root of $A$ to be the largest zero of $z - 2 +\dfrac{2}{3} z R_1(z) - zR_2(z)=z - 2 +\dfrac{2}{3} z S_1(z) - zS_2(z)$, which is $\theta=2$. Here
\[
U=\begin{pmatrix}
3/2\\1\\1/2\\1
\end{pmatrix}, \ \text{and }
V=\begin{pmatrix}
2/3\\2/3\\4/3\\4/3
\end{pmatrix}
\]
%

\end{exam}	

%

\begin{remark}
Consider the case when $\R$ has a word of length strictly less than $p$. Here we define a new collection of repeated words $\tilde{\R}$ as given in Remark~\ref{rem:new_R} that consists of all words of length $p$ that start with words from $\R$. In this case Theorem~\ref{thm:perronvector} gives Perron eigenvector of $A$ in terms of words from $\F$ and $\tilde{\R}$. For example, if $\Sigma=\{0,1\},\ \F=\{0000\}$ and $\R=\{01(2)\}$. Then we look at the following collection $\tilde{\R}=\{0100(2),0101(2),0110(2),0111(2)\}$ and calculate the eigenvectors of $A$ indexed by $\mathcal{L}_3=\{000,001,010,011,100,101,110,111\}$ where $A_{XY}=k(X*Y)$ for $X,Y\in\mathcal{L}_3$.
\end{remark}

\subsubsection*{Perron eigenvectors for a non-negative matrix}
Let $A=[A_{xy}]_{0\le x,y\le q-1}$ be a non-negative irreducible integer matrix of size $q$. Let $\F=\{xy :   A_{xy}=0\}=\{a_1,\dots,a_s\}$ and $\R=\{xy:A_{xy}>1\}=\{r_1,\dots,r_\ell\}$ with multiplicity $m_i$ of $r_i=xy$ given by $A_{xy}$. If $P(z),Q(z)$ denote the matrix functions as given in~\eqref{eq:matrix_P} and~\eqref{eq:Q} for these collections $\F$ and $\R$. Let $R_i(z),S_i(z)$ denote the $i^{th}$ row sum of $P^{-1}(z),Q^{-1}(z)$, respectively. Since all the forbidden words and repeated words are of length 2 here, the expressions for Perron eigenvectors take a simplified form as given in the following result.

\begin{corollary}[Perron eigenvectors for a non-negative matrix]
Let $U=(U_x)_{0\le x\le q-1}$
and $V=(V_x)_{0\le x\le q-1}$ be the vectors defined as 	
\[
U_x=1-\theta\left(\sum\limits_{\substack{i=1\\r_i\text{ ends with }x}}^\ell \left(1-\frac{1}{m_i}\right)R_i(\theta)+\sum\limits_{\substack{j=1\\a_j\text{ ends with }x}}^sR_{\ell+j}(\theta)\right),
\]
\[
V_x=1-\theta\left(\sum\limits_{\substack{i=1\\r_i\text{ begins with }x}}^\ell \left(1-\frac{1}{m_i}\right)S_i(\theta)+\sum\limits_{\substack{j=1\\a_j\text{ begins with }x}}^sS_{\ell+j}(\theta)\right),
\] where $\theta$ is the Perron root of $A$. Then $U$ and $V$ are left and right Perron eigenvectors, respectively, of $A$.
\end{corollary}	

\begin{remark}
Let $A$ be the adjacency matrix associated with the collections $\F$ and $\R$. Then $A$ is indexed by $\mathcal{L}_{p-1}$. Let $\G_A=(\mathcal{L}_{p-1},\E_A)$ denote the graph where $\E_A=\{(X*Y)_i:X,Y\in\mathcal{L}_{p-1},1\le i\le A_{XY}\}$ is the edge set of $\G_A$. Let $\Sigma_A$ denote the edge shift associated to the matrix $A$. Let us now consider a binary matrix $E$, indexed by $\E_A$. For $(X*Y)_i,(W*Z)_j\in\E_A$, $E_{(X*Y)_i,(W*Z)_j}=1$ if and only if $Y=W$. If $\Sigma^v_E$ denotes the vertex shift associated with $E$, then $\Sigma^v_E=\Sigma_A$. Note that $A$ and $E$ have the same Perron root, say $\theta$. If $U,V$ denote the left and right Perron (column) eigenvectors of $A$ and $\hat{U},\hat{V}$ denote the left and right Perron (column) eigenvectors of $E$, respectively, such that $U^TV=1=\hat{U}^T\hat{V}$, then we have for $X\in\mathcal{L}_{p-1}$,
\begin{eqnarray*}
V_X&=& \sum\limits_{Y\in\mathcal{L}_{p-1}}\sum\limits_{1\le i\le k(X*Y)}\hat{V}_{(X*Y)_i},\\
\theta U_X&=& \sum\limits_{Z\in\mathcal{L}_{p-1}}\sum\limits_{1\le i\le k(Z*X)}\hat{U}_{(Z*X)_i}.
\end{eqnarray*}

Hence the Perron root and eigenvectors of $E$ (equivalently, the Perron root and eigenvectors of $A$) can be calculated using the techniques discussed in~\cite{Parry}, since $E$ is irreducible whenever $A$ is irreducible. Here, we get a combinatorial expression for $\theta$ and $\hat{U}$ and $\hat{V}$ in terms of the correlation between the words corresponding to the zero entries in $E$. Note that the zero entries in $E$ correspond to all the pairs of non-adjacent edges in $\G_A$. It is usually a much bigger collection, especially when the size of $\mathcal{L}_{p-1}$ is large. Hence the results presented in this paper are computationally less expensive in most such cases. Other advantages of looking at $A$ instead of $E$ are discussed in~\cite{Williams}. 
\end{remark}


\section{Normalized Perron Eigenvectors of $A$}\label{sec:normalization}

Let $A$ be the adjacency matrix associated with given $\F$ and $\R$ and $p$ denote the length of the longest word in $\F \cup \R$. In this section, we assume that all words in $\R$ are of length $p$. Then $U$ and $V$ as obtained in Theorem~\ref{thm:perronvector} are left and right Perron (column) eigenvectors of $A$. We will obtain a simple expression for the normalizing factor $U^TV$. This will then be used to obtain a combinatorial expression for the Parry measure on the edge shift associated with $A$.

In~\cite{Parry}, subshifts with no repeated words (multiple edges) were considered and the concept of local escape rate was used to compute $U^TV$. We will use similar techniques here. The results in this section require $\Sigma_A$ to be an irreducible subshift with positive topological entropy, that is, the Perron root $\theta$ is strictly bigger than one. We will now define the concept of escape rate in the setting of an edge shift $\Sigma_A$. Let $\mu$ be the Parry measure on $\Sigma_A$.  

\begin{definition}[Escape rate]
Let $W$ be an allowed word in $\Sigma_A$. Consider the cylinder $C_W$ in $\Sigma_A$ based at $W$, which is the collection of all sequences which begin with $W$. The \emph{escape rate} into the hole $C_W$ measures the rate at which the orbits escape into the hole $C_W$ and is defined as 
\[
\rho(C_W) := -\lim_{n\rightarrow \infty} \dfrac{1}{n} \ln\mu(\mathcal{W}_n(W)),
\]
if the limit exists, where $\mathcal{W}_n(W)$ denotes the collection of all sequences in $\Sigma_A$ which do not include $W$ as a subword in their first $n$ positions.
\end{definition} 

In general, the escape rate is defined for any hole of positive measure in $\Sigma_A$, but we only consider that the hole is a cylinder. In the given setting, the limit exists and is given by the following result. Let $h_W(n)$ be the number of allowed words of length $n$ in $\Sigma_A$ that do not contain $W$ as a subword, let $\ln(\lambda_{W})=\lim_{n \to \infty}\frac{1}{n}\ln(h_W(n))$, and let $\theta$ be the Perron root of $A$. Then we have the following result which is a direct generalization of~\cite[Theorem 3.1]{Product}. 

\begin{theorem}\label{thm:esc_rate}
The escape rate into the hole $C_W$ satisfies $\rho(C_W)= \ln (\theta/\lambda_W)>0$. 
\end{theorem}

\noindent For each word $W=(X_1*X_2)_{i_1}(X_2*X_3)_{i_2}\dots(X_{n-p+1}*X_{n-p+2})_{i_{n-p+1}},$ of length $n-p+1$ in $\Sigma_A$, there is a unique word $w=x_1\dots x_n$ of length $n$ in $\Sigma_\F$ where $X_{i} = x_i \dots x_{i+p-2}$. Also, if $m(w)=1$, then $h_W(n-p+1)=\tau_w(n)$, where $\tau_w(n)$ is the number of allowed words of length $n$ in $\Sigma_{\F\cup\{ w\}}$ counted with multiplicity with respect to $\R$. Hence, $\lambda_{W}=\theta_{w}$, where $ \ln \theta_w= \lim_{n \to \infty}\frac{1}{n}\ln \tau_w(n)$ can be calculated using the correlation between the words from $\F\cup\{w\}$ and $\R$ (as in Theorem~\ref{thm:Perron_root}). Note that $m(w)=1$ implies $\F\cup\{w\}\cup\R$ is reduced. 

\begin{remark}
If $w$ contains a repeated word, then $\lambda_W$ and $\theta_w$ may not be equal. 
%
For example, let $A=\begin{pmatrix}
0&2\\1&1
\end{pmatrix}$ and $\G_A=(\{0,1\},\E_A)$ where $\E_A=\{(01)_1,(01)_2,(10)_1,(11)_1\}$. If $W=(01)_2(11)_1$, then $w=011$, and $\tau_w(3)=|\{(011,2),(101,2),(110,1),(111,1)\}|=6$, $ h_W(2)=|\{(01)_1(10)_1,(01)_1(11)_1,(01)_2(10)_1,(10)_1(01)_1,(10)_1(01)_2,(11)_1(10)_1,$ $(11)_1(11)_1\}|$ $=7$.
Moreover, $\lambda_W>\theta_w$.
\end{remark}
\begin{definition}[Local escape rate]
Let $\alpha=\alpha_1\alpha_2\dots\in\Sigma_A$. The \emph{local escape rate} around $\alpha$ is defined as 
\[
\rho(\alpha)=\lim_{n\to\infty}\frac{\rho(C_{W_n})}{\mu(C_{W_n})},
\]
if it exists, where $W_n=\alpha_1\alpha_2\dots \alpha_n$. 
\end{definition}

%
For fixed $X,Y\in\mathcal{L}_{p-1}$, choose a point $\alpha\in\Sigma_A$ such that there exists a subsequence $(n_k)_k$ where for each $k$, $w_{n_k}$ (associated with $W_{n_k-p+1}$) is a word that begins with $X$ and ends with $Y$ (in other words $W_{n_k-p+1}$ is a path from vertex $X$ to vertex $Y$ in the graph $\G_A$). Such an $\alpha$ exists as $A$ is irreducible. We assume that there exist $X,Y\in\mathcal{L}_{p-1}$ such that the word $w_{n_k}$ has multiplicity 1 for all $k$. Using Theorem~\ref{thm:esc_rate} and the expression for the Parry measure, we get that 
\[
\rho(\alpha)=\lim_{k \to \infty}\frac{U^TV\theta^{n_k-p+1}\ln(\theta/\lambda_{W_{n_k-p+1}})}{U_{X}V_{Y}}=\frac{U^TV}{\theta^{p-1}U_{X}V_{Y}}\lim_{k \to \infty}\theta^{n_k}\ln(\theta/\theta_{w_{n_k}}),
\] where $U$ and $V$ are the Perron eigenvectors of $A$ as given in Theorem~\ref{thm:perronvector}.

\begin{remark}\label{rem:non_reduced}
(1) If $\alpha=(X_1*X_2)_{i_1}(X_2*X_3)_{i_2}\dots\in\Sigma_A$ is periodic with period $t$, then the word $\chi=X_1*X_2*X_3*\dots\in \Sigma_\F$ is periodic with period $t$. Using this and the local escape rate formula given by Ferguson and Pollicott in~\cite[Corollary 5.4.]{gibbs}, $\lim_{k \to \infty}\theta^{-k+1}(w_k,w_k)_\theta=\frac{1}{\rho(\alpha)}.$
\\
\end{remark}

\begin{definition}[Property (P)]
We say that the subshift $\Sigma_A$ satisfies \emph{property (P)} if there exist $X,Y\in\mathcal{L}_{p-1}$, a point $\alpha=\alpha_1\alpha_2\dots\in\Sigma_A$, and a strictly increasing sequence $(n_k)_{k\ge 1}$ of natural numbers such that for each $k\ge 1$, $W_{n_k-p+1}=\alpha_1\dots \alpha_{n_k-p+1}$ begins with $X$ and ends with $Y$ and $m(w_{n_k})=1$, where $w_{n_k}$ is the word in $\Sigma_\F$ associated to $W_{n_k-p+1}$. 
\end{definition}

\begin{remark}
For property (P) to be satisfied, it is enough to show the existence of two words $X,Y\in\mathcal{L}_{p-1}$ for which there are two words $Z,W\in\mathcal{L}$, each with multiplicity 1, such that the word $Z$ starts with $X$ and ends with $Y$, and the word $W$ both starts and ends with $Y$. A large family of subshifts satisfy property (P). For instance, if at least one vertex in $\G_A$ has a single loop (in other words, the adjacency matrix $A$ has at least one diagonal entry equal to 1), then $\Sigma_A$ satisfies property (P).   
\end{remark}

%

Let $P(z)$ be the matrix function corresponding to the collections $\F$ and $\R$ as defined in ~\eqref{eq:matrix_P} and let $R_i(z)$ be the $i^{th}$ row sum of $P^{-1}(z)$. Define rational function $R(z)$ as follows:
\begin{equation}
R(z)= z\sum_{i=1}^\ell\left(1-\frac{1}{m_i}\right)R_i(z)-z\sum_{j=1}^sR_{\ell+j}(z),\label{eq:R}
\end{equation} 
By Theorem~\ref{thm:Perron_root}, the Perron root of $A$ is given by the largest real zero of $z-q+R(z)$. Note that $R(z)=\sum DP^{-1}$, where $\sum$ denotes the sum of all the entries of the matrix and $D$ is the diagonal matrix as defined in Equation~\eqref{eqn:diagD}.

Now assume that $\Sigma_A$ satisfies property (P). With notations as in the definition of property (P), $w_{n_k}$ starts with $X$ and ends with $Y$, for all $k\ge 1$. Denote $\F_k=\F\cup\{w_{n_k}\}$, then $\F_{k}\cup\R$ is reduced for all $k\ge 1$. Suppose $P_{k}(z)$ be the matrix function (see~\eqref{eq:matrix_P}) and $R_{k}(z)$ be the rational function (see~\eqref{eq:R}) associated with the collections $\F_k$ and $\R$. Let $\mathcal{D}_k(z)$ and $\mathcal{D}(z)$ denote the determinant of $P_k(z)$ and $P(z)$, respectively. Then 
\[
P_k(z)=\begin{pmatrix}
P(z)&-zY^T(z)\\X(z)D&Z(z)
\end{pmatrix},
\] 
where $Z(z)=(w_{n_k},w_{n_k})_z$, 
\[
X^T(z)=\begin{pmatrix}
(r_1,w_{n_k})_z\\\vdots\\(r_\ell,w_{n_k})_z\\(a_1,w)_z\\\vdots\\(a_s,w_{n_k})_z
\end{pmatrix}=\begin{pmatrix}
(\tilde{r}_1,X)_z\\\vdots\\(\tilde{r}_\ell,X)_z\\(\tilde{a}_1,X)_z\\\vdots\\(\tilde{a}_s,X)_z
\end{pmatrix}\text{ and }\ Y^T(z)=\begin{pmatrix}
(w_{n_k},r_1)_z\\\vdots\\(w_{n_k},r_\ell)_z\\(w_{n_k},a_1)_z\\\vdots\\(w_{n_k},a_s)_z
\end{pmatrix}=\begin{pmatrix}
(Y,r_1)_z\\\vdots\\(Y,r_\ell)_z\\(Y,a_1)_z\\\vdots\\(Y,a_s)_z
\end{pmatrix},\] where for a word $u$, $\tilde{u}$ is the word obtained by removing the first symbol of $u$. 
Observe that if $U_X(z)=1-\sum XDP^{-1}(z)$ and $V_Y(z)=1-\sum Y(DP^{-1})^T(z)$, then $U_X(\theta)$ and $V_Y(\theta)$ are the Perron eigenvectors obtained as in Theorem~\ref{thm:perronvector}. Using the inverse formula for a $2\times2$ block matrix, we get  
\begin{equation}\label{eq:diff}
R_k(z)-R(z)=\frac{-z\mathcal{D}(z)}{\mathcal{D}_k(z)}U_X(z)V_Y(z).
\end{equation}
Then we have the following result, proof of which is similar to the proof of~\cite[Theorem 7.6]{Parry}, using~\eqref{eq:diff} and Remark~\ref{rem:non_reduced}(1).

\begin{theorem}
With the notations as above, if $\Sigma_A$ satisfies property (P), then 
\[
U^TV=\theta^{p-1}(1+R'(\theta)),
\] 	where $R'(\theta)$ is the derivative of the function $R(z)$ (defined in~\eqref{eq:R}) evaluated at $z=\theta$.
\end{theorem}

\begin{exam}
Let $\Sigma=\{0,1\},\F=\{010\},\R=\{100(3)\}$. Here $R(z)=\frac{2 - z}{z^3+z+2}$. The Perron root of the associated adjacency matrix $A$ is $\theta=2$.
The left and right eigenvectors of $A$ is calculated in Example~\ref{ex:evector} and is given as $U=\begin{pmatrix}
3/2\\1\\1/2\\1
\end{pmatrix}$,  and $V=\begin{pmatrix}
2/3\\2/3\\4/3\\4/3
\end{pmatrix}$. Note that $U^TV=\theta^2(1+R'(\theta))=11/3$. 
\end{exam} 

\begin{exam}\label{ex:sparseP}
Let $\Sigma=\{0,1,\dots,q-1\}$, $\F=\{a_1,\dots,a_s\}$ $\R=\{r_1(m_1),\dots,r_\ell(m_\ell)\}$ be such that $|a_i|=|r_i|=2$ satifying $(r_i,r_i)_z=z, (a_j,a_j)_z=z$ and $(r_i, a_j)_z = (a_j, r_i)_z = 0 $ for all $r_i \in \R$, $a_j \in \F$. Further, $ (r_i, r_j)_z = (a_i, a_j)_z = 0$ for all $r_i \neq r_j \in \R$ and $a_i \neq a_j \in \F$. Then the Perron root $\theta$ is the largest real zero of

\[
z-q-\left(\dfrac{\alpha+\ell-s}{z}\right)
\] where $\alpha=m_1+\dots+m_\ell$.
Therefore $\theta=\frac{q^2+\sqrt{q^2+4(\alpha-\ell+s)}}{2}$. 

\noindent For instance, consider $q=4$, $\R=\{10(a),20(b),30(c)\}$, $\F=\emptyset$. Set $\alpha=a+b+c$. Then the Perron root $\theta$ is the largest real zero of \[
z-4-\left(\dfrac{\alpha-3}{z}\right).
\]
Therefore $\theta=2+\sqrt{1+\alpha}$. Also in this case,
\[
U=\dfrac{1}{2+\sqrt{1+\alpha}}\begin{pmatrix}
\alpha-1+\sqrt{1+\alpha}\\ 2+\sqrt{1+\alpha}\\ 2+\sqrt{1+\alpha}\\ 2+\sqrt{1+\alpha}
\end{pmatrix}, \ \ 	V=\dfrac{1}{2+\sqrt{1+\alpha}}\begin{pmatrix}
2+\sqrt{1+\alpha}\\	a+1+\sqrt{1+\alpha}\\ b+1+\sqrt{1+\alpha}\\ c+1+\sqrt{1+\alpha}
\end{pmatrix}.
\]
Since property (P) is satisfied by the adjacency matrix for the corresponding shift space, $U^TV=\theta(1+R'(\theta))$, where $R(z)=\dfrac{\alpha-3}{z}$. Hence
\[
U^TV=\dfrac{5-\alpha+\sqrt{1+\alpha}}{2+\sqrt{1+\alpha}}.
\]	
\end{exam}

This normalization result immediately gives an alternate definition for the Parry measure on $\Sigma_A$ and is stated below as a corollary. 

\begin{corollary}[A combinatorial expression for the Parry measure]\label{cor:parry}
With the notations as above, assume that $\Sigma_A$ satisfies property (P).
If $\ W=(X_1*X_2)_{i_1}\dots(X_n*X_{n+1})_{i_n}$ is an allowed word in $\Sigma_A$, then the Parry measure of the cylinder based at $W$ is given by,
\begin{equation*}
\mu(C_W)=\dfrac{U_{X_1}V_{X_{n+1}}}{\theta^{n+p-1} \left(1+R'(\theta)\right)},
\end{equation*} 
where $U$ and $V$ are as given in equations \eqref{eq:U} and \eqref{eq:V}.
\end{corollary}

\begin{remark}
As discussed in Remark~\ref{rem:non_reduced}(2), the above mentioned techniques cannot be used if $\Sigma_A$ does not satisfy property (P). However, we conjecture that Corollary~\ref{cor:parry} holds true provided $\F\cup\R$ is reduced and all words from $\R$ have length $p$. That is, if $U$ and $V$ denote the Perron eigenvectors of $A$ as given in Theorem~\ref{thm:perronvector}, then $U^TV=\theta^{p-1}(1+R'(\theta))$. We are tempted to give an example where the conjecture is true. Consider the shift $\Sigma_A$ with $A=\begin{pmatrix}
0&2\\3&2
\end{pmatrix}$. Here $\Sigma=\{0,1\},$ $\F=\{00\}$ and $\R=\{01(2),10(3),11(2)\}$. Clearly property (P) is not satisfied as, there does not exist a point in $\Sigma_\F$ that contains no repeated words. Here we get $F(z)=\dfrac{z}{z-q+R(z)}$, where $R(z)=\dfrac{-3(z+2)}{z^2-3}$. Hence the Perron root of $A$ is given by $\theta=1+\sqrt{7}$. Also, by Theorem~\ref{thm:perronvector}, the left and right eigenvalues of $A$ is given by $U=\begin{pmatrix}
\sqrt{7}-1\\2
\end{pmatrix}$ and $V=\begin{pmatrix}
2(\sqrt{7}-2)\\5-\sqrt{7}
\end{pmatrix}$, respectively. Note that $U^TV=\theta(1+R'(\theta))=28-8\sqrt{7}$.
\end{remark}

		\section{$\Sigma_{\F}$ as a factor of $\Sigma_A$}
		\label{sec:conjugacy and measures}
		In this section, we study the edge shift $\Sigma_{\F}$ as a factor of the edge shift $\Sigma_A$ and discuss some properties of Markov measures on these spaces. Let $\Sigma, \F , \R$ and $p$ as before and $A$ be the adjacency adjacency matrix associated with $\F$ and $\R$. Let $\hat{A}$ be the binary matrix indexed by $\mathcal{L}_{p-1}$, \textit{compatible with $A$}, that is, $\hat{A}_{XY}=1$ if and only if $A_{XY}>0$. Let $\Sigma_A$ and $\Sigma_{\hat{A}}$ be the edge shifts associated with the matrices $A$ and $\hat{A}$, respectively. Note that $\Sigma_{\hat{A}} = \Sigma_{\F}^{[p]}$. Observe that if $\R$ is an empty collection, then $A$ and $\hat{A}$ are the same. However, it was observed earlier, that $\Sigma_A$ and $\Sigma_{\hat{A}}$ are not conjugate to each other if the collection $\R$ is non-empty.
		
		 Let $\G = (\mathcal{V}, \mathcal{E})$ and $\hat{\G} = (\hat{\mathcal{V}}, \hat{\mathcal{E}})$ be the digraphs associated to the matrices $A$ and $\hat{A}$, respectively, both having the same vertex sets, $\mathcal{V} = \hat{\mathcal{V}} = \mathcal{L}_{p-1}$, but (possibly) different edge sets. For $X,Y \in \mathcal{L}_{p-1}$, let $\mathcal{E}_{XY}$ and $\hat{\mathcal{E}}_{XY}$ denote the set of edges from $X$ to $Y$ in $\G$ and $\hat{\G}$, respectively. If $X *Y\in\mathcal{L}_p$, then $\hat{\mathcal{E}}_{XY}=\{X*Y\}$ and $\mathcal{E}_{XY} = \left\lbrace (X * Y)_{j} \, : \, 1 \le j \le k(X * Y) \right\rbrace$, else, $\hat{\mathcal{E}}_{XY}=\mathcal{E}_{XY}=\emptyset$.

		
		\begin{definition}[Projection map for $\G$ and $\hat{\G}$]
			With notations as above, a projection map $\pi: \G \longrightarrow \hat{\G}$ is defined as follows. On vertex sets, $\pi(\mathcal{V}) = \hat{\mathcal{V}} = \mathcal{V} $ is an identity map and on edge sets, $\pi (\mathcal{E}) = \hat{\mathcal{E}}$, with $\pi (\mathcal{E}_{XY}) = \hat{\mathcal{E}}_{XY}$ for all $X,Y \in \mathcal{V}$. 
		\end{definition}
		
		The projection map $\pi$ identifies all the edges in $\G$ from vertex $X$ to $Y$ as a single edge. Whenever $X*Y$ is allowed, all the edges from $X$ to $Y$ in $\G$ are projected onto the only edge $X$ to $Y$ in $\hat{\G}$.
		
		\begin{exam}
			We look at a simple example to understand the map $\pi$. Let $q = 2$, $\Sigma = \{ 0, \, 1\}$. Let $\mathcal{F} = \{ 11 \}$ and a matrix $A = \begin{pmatrix}
				3 & 1 \\ 2 & 0
			\end{pmatrix} $. Then $\hat{A} = \begin{pmatrix}
				1 & 1 \\ 1 & 0
			\end{pmatrix} $. Consider the digraphs $\G$ and $\hat{\G}$ defined by the adjacency matrices $A$ and $\hat{A}$ as depicted in Figure \ref{fig:G_and_hatG}. The vertex set for both the graphs is $ \mathcal{V} = \{ 0, \, 1\}$. Let us name the edges in $\G$ as $a, b, \dots, f$ as shown in the figure. Then $\pi (\{ a,b,c \}) = \{ 00\}$, $\pi (\{ d \}) = \{ 01\}$, $\pi (\{ e,\, f \}) = \{ 10 \}$.
			
			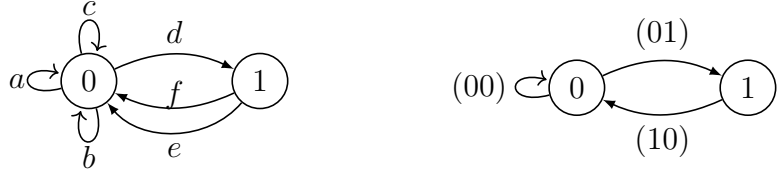
\begin{figure}[h]
				\centering
				$\displaystyle
				\begin {tikzpicture}[-latex ,auto ,node distance =1.5cm and 2.25cm ,on grid ,
				semithick ,
				state/.style ={draw, circle}] 
				\node[state] (A) {$0$};
				\node[state] (B) [right =of A] {$1$};
				
				\path (A) edge [loop left] node[right=-.4cm] {$a$} (A);
				\path (A) edge [loop below] node[above=-.5cm] {$b$} (A);
				\path (A) edge [loop above] node[below=-.4cm] {$c$} (A);
				\path (A) edge [bend left =25] node[above] {$d$} (B);
				\path (B) edge [bend right =-50] node[below] {$e$} (A);
				\path (B) edge [bend right = -25] node[below =-.5 cm] {$f$} (A);
			\end{tikzpicture}
			\hspace{2cm}
			\begin {tikzpicture}[-latex, baseline=-1.2cm, auto ,node distance =1.5cm and 2.25cm ,on grid ,
			semithick ,
			state/.style ={draw, circle}] 
			\node[state] (A) {$0$};
			\node[state] (B) [right =of A] {$1$};
			
			\path (A) edge [loop left] node[right=-1cm] {$(00)$} (A);
			\path (A) edge [bend left =25] node[above] {$(01)$} (B);
			\path (B) edge [bend right = -25] node[below] {$(10)$} (A);
		\end{tikzpicture}
		$
		\caption{Graph $\G$ on the left and graph $\hat{\G}$ on the right}
		\label{fig:G_and_hatG}
	\end{figure}	
\end{exam}

The definition of the projection map $\pi$ can be extended to finite paths as well. A word $W = (X_{1} * X_{2})_{i_{1}} (X_{2} * X_{3})_{i_{2}} \dots  (X_{n} * X_{n+1})_{i_{n}}$ denotes a path from $X_{1}$ to $X_{n+1}$ of length $n$ in $\G$. As $i_{i}, i_{2}, \dots, i_{n}$ vary over all the possible choices, we get all the paths from $X_{1}$ to $X_{n+1}$ with the fixed intermediate vertices (in order) $X_{2}, \dots, X_{n}$. We define $\pi$ to map all these different paths in $\G$ onto a fixed path in $\hat{\G}$ as,
\[ \pi \left( (X_{1} * X_{2})_{i_{1}} (X_{2} * X_{3})_{i_{2}} \dots  (X_{n} * X_{n+1})_{i_{n}} \right) := (X_{1} * X_{2})(X_{2} * X_{3}) \dots (X_{n} * X_{n+1}),\]
for all $1 \le i_{j} \le k(X_{j}*X_{j+1})$, $1 \le j \le n$.

The shifts $\Sigma_A$ and $\Sigma_{\F}^{[p]}$ represent the spaces of all one-sided infinite paths on the graphs $\G$ and $\hat{\G}$, respectively. The map $\pi$ can be further extended to $ \pi : \Sigma_A \longrightarrow \Sigma_{\F}^{[p]}$ as follows.
\begin{equation}
	\label{eq:projection}
	\pi \left( (X_{1} * X_{2})_{i_{1}} (X_{2} * X_{3})_{i_{2}} (X_{3} * X_{4})_{i_{3}} \dots \right) :=  (X_{1} * X_{2}) (X_{2} * X_{3}) (X_{3} * X_{4}) \dots \ ,
\end{equation}
for all $1 \le i_{j} \le k(X_{j}*X_{j+1})$, $j \ge 1$.


\begin{proposition}
	The shift $\Sigma_{\F}$ is a factor of the edge shift $\Sigma_A$ with factor map $\pi$.
\end{proposition}
\begin{proof}
	\begin{equation*}
		\label{eq:ch1_commuting_diag} 
		\begin{tikzcd}
			\Sigma_A\arrow[r, "\sigma_A"] \arrow[d, "\pi"] & \Sigma_A						\arrow[d, "\pi"] \\
			\Sigma_{\mathcal{F}}^{[p]} \arrow[r, "\sigma"] & 									\Sigma_{\mathcal{F}}^{[p]}
		\end{tikzcd}
	\end{equation*}
	We first recall that $\pi : \Sigma_A \longrightarrow \Sigma_{\F}$ is called a {\it factor map} or {\it projection} if $\pi$ is onto and the above diagram commutes. The proof now follows directly from the definitions of $\pi$ as given in \eqref{eq:projection} and of the left shift maps $\sigma_A$ and $\sigma$.
\end{proof}

We now recall some basic concepts related to a Markov measure on a shift space. A square matrix is called a {\it stochastic matrix} if it is non-negative, with each row sum being $1$. Let $\hat{\mathcal{P}}$ be a stochastic matrix compatible with $\hat{A}$. Then $1$ is the Perron root of $\hat{\mathcal{P}}$ with right Perron eigenvector given by $(1,1,\dots,1)^T$. Let $\hat{\rho}$ be the normalized left Perron eigenvector of $\hat{\mathcal{P}}$. A {\it Markov chain} corresponding to $\hat{\mathcal{P}}$ on the graph $\hat{\G}$ is the assignment of probability to the vertices given by $\nu(X) := \hat{\rho}_{X}$, and to the edges given by $\nu(X*Y | X) := \hat{\mathcal{P}}_{XY} $ which is the conditional probability of the edge $X*Y$. The probability of the edge $X*Y$ is then given by $\nu(X*Y) = \hat{\rho}_{X} \hat{\mathcal{P}}_{XY}$. The probability of a path is expressed as the product of probability of the initial vertex and conditional probabilities of the edges in the path.   

Let $\hat{W} = (X_{1} * X_{2}) (X_{2} * X_{3}) \dots  (X_{n} * X_{n+1})$ be an allowed word of length $n$ in $\Sigma_{\mathcal{F}}^{[p]}$ and let $C_{\hat{W}}$ be the cylinder in $\Sigma_{\mathcal{F}}^{[p]}$ based at $\hat{W}$. The stochastic matrix $\hat{\mathcal{P}}$ defines a \emph{Markov measure} $\nu$ on $\Sigma_{\F}^{[p]}$ which is described for the cylinder sets as follows,
\[  \nu (C_{\hat{W}}) = \hat{\rho}_{X_{1}} \, \hat{\mathcal{P}}_{X_{1}X_{2}} \, \hat{\mathcal{P}}_{X_{2}X_{3}} \, \dots \, \hat{\mathcal{P}}_{X_{n}X_{n+1}}.
\]

The Parry measure on $\Sigma_{\F}^{[p]}$ defined in Section~\ref{sec:SFT_binary} is a special kind of Markov measure. The stochastic matrix associated with it is called the \emph{Parry matrix}, and is given as follows.
\begin{definition}
	The Parry matrix $\hat{\mathcal{P}}$ associated to a binary matrix $\hat{A}$ is a stochastic matrix compatible with $\hat{A}$ and is defined as
	\begin{equation*}
		\label{eq:Parry matrix for binary}
		\hat{\mathcal{P}}_{XY} := \frac{\hat{A}_{XY}\, \hat{V}_{Y}}{\hat{\theta} \, 			\hat{V}_{X}}, \qquad 	\text{for } X, Y \in \mathcal{L}_{p-1},
	\end{equation*}
	where $\hat{\theta}$ is the Perron root and $\hat{U}$ and $\hat{V}$ are the left and right Perron (column) eigenvectors of $\hat{A}$ such that $\hat{U}^T\hat{V}=1$. The normalized left Perron eigenvector $\hat{\rho}$ of $\hat{\mathcal{P}}$ is given by
	\[ 
	\hat{\rho}_{X} = \hat{U}_{X}\hat{V}_{X}.
	\]
\end{definition}

Let $A$ be the adjacency matrix associated with $\F$ and $\R$. A Markov measure on $\Sigma_A$ is defined in a similar way.  Let $\mathcal{P}$ be a stochastic matrix compatible with $A$ and $\rho$ be its normalized left Perron eigenvector. A \textit{Markov chain} $\mu$ associated to $\mathcal{P}$ on the graph $\G$ (associated to matrix $A$) is given by $\mu(X) = \rho_{X}$ and $\sum\limits_{1 \le i \le k(X*Y)} \mu((X*Y)_{i}|X) = \mathcal{P}_{XY}$. Here the sum of the conditional probabilities of all the edges from $X$ to $Y$ is the $XY^{th}$ entry of $\mathcal{P}$. Let $C_{W}$ be the cylinder in $\Sigma_A$ based at a word $ W = (X_{1} * X_{2})_{i_{1}} (X_{2} * X_{3})_{i_{2}} \dots  (X_{n} * X_{n+1})_{i_{n}}$ in $\Sigma_A$. The stochastic matrix $\mathcal{P}$ gives rise to a Markov measure on $\Sigma_A$, which is described for cylinder sets as follows: for a cylinder $C_W$,
\[	\mu (C_{W}) = \mu(X_{1}) \, \mu((X_{1} * X_{2})_{i_{i}}\, \big| \,X_{1}) \, \mu((X_{2} * X_{3})_{i_{2}}\, \big| \,X_{2}) \, \dots \, \mu((X_{n} * X_{n+1})_{i_{n}}\, \big| \,X_{n}).
\]

Consider the preimage of the word $\hat{W}= (X_{1} * X_{2}) \dots  (X_{n} * X_{n+1})$ in $\Sigma_{\F}^{[p]}$ under $\pi$, which is the set of all paths from $X_{1}$ to $X_{n+1}$ with the fixed intermediate vertices $X_{2}, \dots, X_{n}$, in order, in the graph $\G$. That is,
\[ \pi^{-1}(\hat{W}) = \left\lbrace (X_{1} * X_{2})_{i_{1}} \dots  (X_{n} * X_{n+1})_{i_{n}} \, \big| \,  1 \le i_{j} \le k(X_{j}*X_{j+1}), \ 1 \le j \le n  \right\rbrace.
\]
The Markov measure of the set of all cylinder sets based at words $W \in \pi^{-1}(\hat{W})$ is given as
\begin{align*}
	\mu \left( \bigcup\limits_{W \in \pi^{-1}(\hat{W})}  C_{W} \right) & =  \sum\limits_{W \, \in \, \pi^{-1}(\hat{W})} \mu \left( C_{W} \right) \\
	& = \sum\limits_{i_{1} = 1}^{k(X_{1}*X_{2})} \dots \sum\limits_{i_{n} = 1}^{k(X_{n}*X_{n+1})}  \mu(X_{1}) \mu((X_{1} * X_{2})_{i_{i}}\big| X_{1}) \dots  \mu((X_{n} * X_{n+1})_{i_{n}} \big| X_{n})  \\
	& = \rho_{X_{1}}\, \mathcal{P}_{X_{1}X_{2}} \, \mathcal{P}_{X_{2}X_{3}} \, \dots \, \mathcal{P}_{X_{n}X_{n+1}},
\end{align*}
since the union on the left hand side above is a disjoint union of cylinder sets $C_{W}$ for all $W \in \pi^{-1}(\hat{W})$.

The Shannon-Parry measure on $\Sigma_A$ as defined in Definition~\ref{def:sp} is also a Markov measure, very similar to the Parry measure on $\Sigma_{\F}^{[p]}$. The corresponding stochastic matrix is known as the \textit{Shannon-Parry matrix} and is defined below. Let $\theta$ be the Perron eigenvalue and $U$ and $V$ be the corresponding right and left Perron (column) eigenvectors of $A$ such that $U^{T}V = 1$. 
\begin{definition}
	The Shannon-Parry matrix $\mathcal{P}$ associated with $A$ and the normalized left Perron eigenvector $\rho$ of $\mathcal{P}$ are defined as
	\begin{equation*}
		\begin{split}
			\label{eq:Parry matrix}
			\mathcal{P}_{XY} &:= \frac{A_{XY}\, V_{Y}}{\theta \, V_{X}}, \qquad 				\text{for } X, Y \in \mathcal{L}_{p-1} \\
			\rho_{X} &:= U_{X}V_{X}.
		\end{split}
	\end{equation*}
\end{definition}

\begin{remark}
	Let $ W = (X_{1} * X_{2})_{i_{1}} (X_{2} * X_{3})_{i_{2}} \dots  (X_{n} * X_{n+1})_{i_{n}}$ be a word of length $n$ in $\Sigma_A$. The Shannon-Parry measure $\mu$ of $C_{W}$ is given in terms of the Shannon-Parry matrix $\mathcal{P}$ as 
	\begin{equation*}
		\label{eq:Parry_cylinder}
		\mu (C_{W}) = \rho_{X_{1}} \, \frac{\mathcal{P}_{X_{1}X_{2}}}{A_{X_{1}X_{2}}} \,\frac{\mathcal{P}_{X_{2}X_{3}}}{A_{X_{2}X_{3}}} \, \dots \frac{\mathcal{P}_{X_{n}X_{n+1}}}{A_{X_{n}X_{n+1}}}.
	\end{equation*}
	Note that if $W' = (X_{1} * X_{2})_{j_{1}} (X_{2} * X_{3})_{j_{2}} \dots  (X_{n} * X_{n+1})_{j_{n}}$ is another path of length $n$ in $\G$, then $\mu (C_{W}) = \mu (C_{W'})$. In this case, all the edges from two fixed vertices are given equal conditional probabilities which need not be the case for a general Markov measure on $\Sigma_A$. 
\end{remark}
%
%
%

\begin{theorem}
	Any Markov measure on $\Sigma_A$ induces a Markov measure on $\Sigma_{\F}^{[p]}$ as a push forward measure under $\pi$. That is, for every Markov measure $\mu$ on $\Sigma_A$, there exists a Markov measure $\nu$ on $\Sigma_{\F}^{[p]}$ such that for any cylinder $C_{\hat{W}}$ in $\Sigma_{\F}^{[p]}$ based at a word $\hat{W} \in \mathcal{L}$, 
	\[  \nu(C_{\hat{W}}) = \mu \left( \pi^{-1}(C_{\hat{W}}) \right).
	\]
\end{theorem}

\begin{proof}
	Let $\hat{W} = (X_{1} * X_{2}) (X_{2} * X_{3}) \dots  (X_{n} * X_{n+1})$ be a word of length $n$ in $\Sigma_{\mathcal{F}}^{[p]}$ and let $C_{\hat{W}}$ be the cylinder in $\Sigma_{\mathcal{F}}^{[p]}$ based at $\hat{W}$. Let $\mathcal{P}$ be any stochastic matrix compatible with $A$ and let $\mu$ be the corresponding Markov measure on $\Sigma_A$. 
	Observe that $\mathcal{P}$ is a stochastic matrix compatible with $\hat{A}$ as well, and thus it defines a Markov measure $\nu$ on $\Sigma_{\F}^{[p]}$ as
	\[ \nu(C_{\hat{W}}) = \rho_{X_{1}}\, \mathcal{P}_{X_{1}X_{2}} \, \mathcal{P}_{X_{2}X_{3}} \, \dots \, \mathcal{P}_{X_{n}X_{n+1}}. \]
	
	Note that $\mu \left( \pi^{-1}(C_{\hat{W}}) \right) = \mu \left( \bigcup\limits_{W \, \in \, \pi^{-1}(\hat{W})}  \, C_{W} \right) = \rho_{X_{1}}\, \mathcal{P}_{X_{1}X_{2}} \, \mathcal{P}_{X_{2}X_{3}} \, \dots \, \mathcal{P}_{X_{n}X_{n+1}}$, which is same as $\nu(C_{\hat{W}})$. Hence the result.
	%
	%
	%
	%
	%
\end{proof}
Since the Shannon-Parry measure on $\Sigma_A$ is a Markov measure, we obtain the following corollary.
\begin{corollary}
	The Shannon-Parry measure on $\Sigma_A$ induces a Markov measure on $\Sigma_{\F}$ as a push forward measure under the projection $\pi$.
\end{corollary}

As evident from the above proof, a sufficient condition for a Markov measure on $\Sigma_{\F}^{[p]}$ to be a push forward of the Shannon-Parry measure on some $\Sigma_A$ under $\pi$, is that the stochastic matrices corresponding the two measures should be the same. An intriguing question now arises is that, given any Markov measure on $\Sigma_{\F}^{[p]}$, can it be obtained as a push forward of a Shannon-Parry measure on $\Sigma_A$, for some non-negative irreducible integer matrix $A$. We show that this is not true in general. We further obtain certain conditions on the Markov measure under which this holds true.

\begin{theorem}
	\label{thm:converse_pushforward1}
	Suppose $\nu$ is a Markov measure on $\Sigma_{\F}^{[p]}$ such that the associated stochastic matrix $\mathcal{P}_{\nu}$ of $\nu$ has all rational entries. Then there exists a non-negative irreducible integer matrix $A$ such that $\nu$ is the push-forward of the Shannon-Parry measure on $\Sigma_A$.
\end{theorem}

\begin{proof}
	Given that $\mathcal{P}_{\nu}$ is a rational matrix, any nonzero $XY^{th}$ entry of $\mathcal{P}_{\nu}$ is $\frac{n_{XY}}{d_{XY}}$ (reduced form). Set $L = \lcm  \left\lbrace d_{XY} \, : \, ({\mathcal{P}_{\nu}})_{XY} > 0 \right\rbrace$. It can be verified with simple calculations that the matrix $A := L \, \mathcal{P}_{\nu}$ is the required matrix. Here $\mathcal{P}_{\nu}$ itself is the Shannon-Parry matrix for $A$. We use the fact that $1$ is the Perron root of a stochastic matrix with corresponding right eigenvector $(1,1,\dots,1)$.
	
	%
	
\end{proof}

\begin{exam}
	We give an example to show that the hypothesis given in Theorem~\ref{thm:converse_pushforward1} is not necessary.
	Let $\Sigma=\{0,1\}$, and $\F=\{11\}$. We give two examples of Markov measures $\nu_1$ and $\nu_2$ on $\Sigma_\F^{[2]}$ with associated stochastic matrices $\mathcal{P}_{1}$ and $\mathcal{P}_{2}$, respectively, having irrational entries. We prove that $\nu_1$ is the push-forward for the Shannon-Parry measure on $\Sigma_A$ for some $A$ whereas $\nu_2$ is not. \\	
	\begin{enumerate}
		\item Let  $\mathcal{P}_1=\begin{pmatrix}
			2(\sqrt{2}-1)&3-2\sqrt{2}\\1&0
		\end{pmatrix}$. Note that $\nu_1$ is a push-forward of the Shannon-Parry measure on $\Sigma_A$ where $A=\begin{pmatrix}
			2&1\\1&0
		\end{pmatrix}$.
		\item Let $\mathcal{P}_2=\begin{pmatrix}
			1/\pi&1-1/\pi\\1&0
		\end{pmatrix}$. If there exists a non-negative integer matrix $A=\begin{pmatrix}
			A_{11}&A_{12}\\A_{21}&A_{22}
		\end{pmatrix}$ such that $\mathcal{P}_{XY}=\frac{A_{XY}V_Y}{\theta V_X}$ where $\theta$ is the Perron root of $A$ and $V$ is the right Perron eigenvector of $A$. We have, in particular, $\frac{1}{\pi}=\mathcal{P}_{11}=\frac{A_{11}}{\theta}$ which gives $\theta=A_{11}\pi$. Note that $\theta$ is a algebraic number as it satisfies the characteristic equation of $A$ which gives a contradiction. 
	\end{enumerate}
\end{exam}

\begin{theorem}
	\label{thm:converse_pushforward2}
	Let $\F$ be a given collection of forbidden words and $p$ be the length of the longest word in $\F$. Set $\R = \mathcal{L}_{p}$, such that all words in $\R$ are assigned equal multiplicities, say $M$. Let $A$ be the adjacency matrix associated with $\F$ and $\R$, and $\Sigma_A$ be the corresponding edge shift. Then the Parry measure on $\Sigma_{\F}^{[p]}$ is the push forward of the Shannon-Parry measure on $\Sigma_A$.
\end{theorem}

\begin{proof}
	Let $\hat{A}$ be the binary matrix compatible with $A$ and $\hat{\mathcal{P}}$ be the associated Parry matrix. It is given that $A = M \hat{A}$. If $\mathcal{P}$ denotes the Shannon-Parry matrix associated to $A$, then simple calculations show that $\hat{\mathcal{P}} = \mathcal{P}$. That is, the same stochastic matrix induces the Parry measure on $\Sigma_{\mathcal{F}}^{[p]}$ and the Shannon-Parry measure on $\Sigma_A$. Hence the result follows. 
	%
	%
\end{proof}

\begin{exam}
	We give an example to show that the hypothesis given in Theorem~\ref{thm:converse_pushforward2} is not necessary. Let $\Sigma=\{0,1\}$ and $\F=\{11\}$. The adjacency matrix of $\Sigma_\F^{[2]}$ is given by $\hat{A}=\begin{pmatrix}
		1&1\\1&0
	\end{pmatrix}$ which has Perron root $\hat{\theta}=(1+\sqrt{5})/2$ and right Perron eigenvector $\hat{V}=\begin{pmatrix}
		(1+\sqrt{5})/2\\1
	\end{pmatrix}$. Let $\R=\{00(2),10(4)\}$. Note that $\R$ does not satisfy the conditions as in Theorem~\ref{thm:converse_pushforward2}. The adjacency matrix associated with $\F$ and $\R$ are $A=\begin{pmatrix}
		2&1\\4&0 
	\end{pmatrix}$ which has Perron root $\theta=1+\sqrt{5}$ and right Perron eigenvector $V=\begin{pmatrix}
		(1+\sqrt{5})/4\\1
	\end{pmatrix}$. Note that for $X,Y\in\{0,1\},$ $\frac{A_{XY}V_Y}{V_X}=\frac{\hat{A}_{XY}\hat{V}_Y}{\hat{V}_X}.$  
\end{exam}

\section{Concluding remarks and Potential questions}\label{sec:conclusion}
\noindent In this section, we discuss some immediate consequences of our work and also potential questions that emerge from this study.

The methods proposed in this paper for computing the Perron root and eigenvectors (Theorems~\ref{thm:Perron_root},\ref{thm:perronvector}) are particularly useful when the matrices $P,Q$ are sparse. That is, a lot of pairs of distinct words in $\F\cup\R$ have zero cross-correlations (i.e., $(u,w)_z=0$ for all $u\ne w\in\F\cup\R$). For instance, if all the cross-correlations are zero, then the matrices $P,Q$ take the form of a diagonal matrix. Hence it is much easier to use our methods than the traditional methods.
Here the rational function $F(z)$ of which the Perron root $\theta$ of the adjacency matrix $A$ is the largest real pole and also the expressions of the Perron eigenvectors take a simple form since 
\[
R_i(z)=Q_i(z)=\begin{cases}
	\left(z\left(1-\dfrac{1}{m_i}\right)(r_i,r_i)_z-z^{|r_i|}\right)^{-1}, & 1\le i\le \ell,\\
	\left(-z(a_i,a_i)_z\right)^{-1}, & \ell+1\le i\le \ell+s.
\end{cases}
\]
For an illustration, revisit Example~\ref{ex:sparseP}. Moreover, this method is extremely useful when the matrix $A_{n\times n}$ differs from $\mathbbm{1}_{n\times n}$ (the matrix where all entries are 1) only on a few places, in which case, the matrices $P,Q$ are small and the calculations are simpler in terms on computational complexity. Also, these algorithms do not depend on the size of $A$ hence may have applications in network theory where usually matrices of large order are studied. Another interesting observation about our work is that the expressions of the Perron root and eigenvectors only depend on the correlation between the words corresponding to zero or greater than 1 entries in the matrix, which implies that when the correlations are the same, no matter where they are located in the matrix, these quantities are the same.

For a matrix $A$ with non-negative rational entries, choose the smallest $L>0$ so that $LA$ is an integer matrix. Then if $\theta,U,V$ are the Perron root and left and right Perron eigenvectors of $A$, respectively, then $L\theta,U,V$ are the Perron root and left and right Perron eigenvectors of $LA$, respectively. Thus our results are easily extendable to matrices with non-negative rational entries. 

%
%
	%
	Further, since the recurrence relations obtained in~\cite{Guibas} (when the collection $\R$ is empty) are useful in non-transitive games, period prefix-synchronized coding, information theory or network theory, we believe that the generalized recurrence relations obtained in Section~\ref{sec:recurrence relations} may have several other independent applications. 
	
	There are some immediate potential questions that emerge from this study such as obtaining an expression for the Perron eigenvectors when $\R$ has words of different lengths or when $\F\cup\R$ is not reduced or proving the normalization result, as discussed in Section~\ref{sec:normalization}, when property (P) is not satisfied. 
	Moreover, this work leads to the question of characterizing collections $\F$ and $\R$ that maximize the entropy of the shift $\Sigma_A$.

	\addcontentsline{toc}{chapter}{References}
	\bibliographystyle{abbrv}  
	\renewcommand{\bibname}{References} 
	\bibliography{mybib} 
	
	\appendix
	\section*{Appendix}\label{sec:app}
	\begin{proof}[Proof of Theorem~\ref{thm:non_reduced_gen_fun}]
		To obtain the first equation, consider the word $w x$ obtained by adjoining a word $w$ in $f(n)$, and a symbol $x \in \Sigma$. Then, by similar arguments as in the proof of Theorem~\ref{thm:reduced_gen_fun}, we get 
		\begin{equation*}
			\sum_{x \in \Sigma} \ \sum_{w\in\mathcal{L}_n} m(wx) \ = \ q  f(n)  +  \sum\limits_{i=1}^{\ell}  \frac{g_{r_{i}}(n+1)}{m_{i}} (m_{i}-1). 
		\end{equation*}
		A word $w x$ gets counted either in $f(n+1)$ or in $f_{a_{i}}(n+1)$ with extra counting and thus
		\begin{equation}
			\label{eq:non_reduced_step1}
			q  f(n)  +  \sum\limits_{i=1}^{\ell}  \frac{g_{r_{i}}(n+1)}{m_{i}} (m_{i}-1) = f(n+1) + \sum\limits_{i = 1}^{s} f_{a_{i}}(n+1)\  + \ T(n).
		\end{equation}
		Here the term $T(n)$ appears only when $w x$ ends with a forbidden word $a_{i}$, such that $a_{i}$ contains a repeated word $r_{j}$, except at the end. In this case, $\alpha \in (a_{i}, r_{j})$ with $\alpha > |r_{j}|$. This implies $r_{j}$ is a subword of $w$ and thus $w$ is counted $m_j$ times in $q f(n)$ in the LHS of~\eqref{eq:non_reduced_step1}. Further, in the RHS of \eqref{eq:non_reduced_step1}, $w$ is counted only once in $f_{a_{i}}(n+1)$, since $r_{j}$ being a subword of $a_{i}$ does not contribute to any extra counting. Thus we need to add $(m_{j}-1)f_{a_{i}}(n+1)$ in all such cases. The term $T(n)$ refers to all such extra counting. It is given by 
		\[
		T(n)= \sum\limits_{i=1}^{s}  \ \sum\limits_{j = 1}						^{\ell} \sum\limits_{\substack{{ 0 \, < \, \alpha \, \in \, (a_{i}, 				r_{j})}\\{\alpha \, > \,|r_{j}|} }} (m_{j} -1) f_{a_{i}}(n+1).
		\] 
		Using the definition of $\gamma(a_{i}, r_{j})$ as given in~\eqref{eq:gamma}, we obtain the following
		\begin{equation}\label{eq:non_reduced_step1_final}
			\begin{split}
				q f(n) -  f(n+1) &=  \sum\limits_{i = 1}^{s} f_{a_{i}}(n+1) -\sum\limits_{i=1}^{\ell} \left( 1 -\frac{1}{m_{i}} \right) g_{r_{i}}(n					+1) \\
				& \ \qquad + \ \sum\limits_{i=1}^{s}  \ \left[ \sum\limits_{j = 1}^{\ell} \gamma(a_{i}, r_{j}) (m_{j} -1) \right] \, f_{a_{i}}(n+1).
			\end{split}
		\end{equation}
		Multiplying the above equation by $z^{-n}$ and taking the sum over $n \ge 0$, we get~\eqref{eq:non_reduced1} as required.
		
		In order to obtain the second equation, we adjoin a repeated word $r_{k}$ to an allowed word $w$ counted in $f(n)$ as a suffix. Our claim is that $f(n)$ satisfies
		\begin{equation}\label{eq:non_reduced_step2_final}
			\begin{split}
				f(n) & =  g_{r_{k}}(n+|r_{k}|) - \sum\limits_{j=1}^{\ell} \ \sum			\limits_{0 \, < \, s \, \in \, (r_{j}, r_{k})} \left( 1- \frac{1}					{m_{j}} \right) g_{r_{j}}(n+s)  + \sum\limits_{i=1}^{s} \ \sum\limits_{0 < t \in (a_{i}, r_{k})^{|r_{k}|}} f_{a_{i}}(n+t)  \\
				&  + \ \sum\limits_{i=1}^{s} \ \sum\limits_{0 \, < \, t \,\in \, (a_{i}, r_{k})} \ \sum\limits_{j = 1}^{\ell} \gamma(a_{i},r_{j}) (m_{j} -1) f_{a_{i}}(n+t).
			\end{split}		
		\end{equation}
		Two possible situations arise, as described below:\\
		1) If no forbidden word appears on the join of $w$ and $r_{k}$, then $w$ is counted in $g_{r_{k}}(n + |r_{k}|)$, the first term on the RHS of~\eqref{eq:non_reduced_step2_final}. Further, if a repeated word $r_{j}$ appears on the join, then there exists $0 < s \in (r_{j}, r_{k})$ such that $w$ is counted extra $m_{j} -1$ times in $g_{r_{k}}(n + |r_{k}|)$. The total number of $w$ in $f(n)$ which give $r_{j}$ on the join is $\frac{g_{r_{j}}(n+s)}{m_{j}}$, for some $0 < s \in (r_{j}, r_{k})$. We subtract these extra counting to obtain the second term on the RHS of~\eqref{eq:non_reduced_step2_final}.\\
		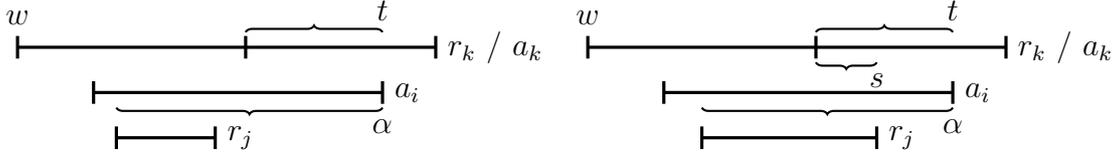
\begin{figure}
			\begin{center}
				\begin{tikzpicture}
					\beginpgfgraphicnamed{situation-a}
					\begin{scope}[very thick]
						\draw[snake=brace] (3,0.2) -- (4.8,0.2) [thick] node[above] 							{$t$};
						\draw (0,-0.15) -- (0,0.15)  node[anchor=south] {$w$} ;
						\draw (0,0) -- (3,0) -- (5.5,0)  node[anchor=west] {$r_{k}$ /	$a_{k}$};
						\draw (3,-0.15) -- (3,0.15);
						\draw (5.5,-0.15) -- (5.5,0.15);
						\draw[snake=brace, mirror snake] (1.3,-0.8) -- (4.8,-0.8) 								[thick] node[below] {$\alpha$};
						\draw (1,-0.6) -- (4.8,-0.6) node[anchor=west] {$a_{i}$};
						\draw (1,-0.75) -- (1,-0.45);
						\draw (4.8,-0.75) -- (4.8,-0.45);
						\draw (1.3,-1.2) -- (2.6,-1.2) node[anchor=west] {$r_{j}$};
						\draw (1.3,-1.05) -- (1.3,-1.35);
						\draw (2.6,-1.05) -- (2.6,-1.35)  ;
						\draw[snake=brace] (10.5,0.2) -- (12.3,0.2) [thick] 									node[above] {$t$};
						\draw (7.5,-0.15) -- (7.5,0.15)  node[anchor=south] {$w$} ;
						\draw (7.5,0) -- (10.5,0) -- (13,0)  node[anchor=west] 								{$r_{k}$ / $a_{k}$};
						\draw (10.5,-0.15) -- (10.5,0.15);
						\draw (13,-0.15) -- (13,0.15);
						\draw[snake=brace, mirror snake] (10.5,-0.2) -- (11.3,-0.2) 							[thick] node[below] {$s$};
						\draw[snake=brace, mirror snake] (9,-0.8) -- (12.3,-0.8) 								[thick] node[below] {$\alpha$};
						\draw (8.5,-0.6) -- (12.3,-0.6) node[anchor=west] {$a_{i}$};
						\draw (8.5,-0.75) -- (8.5,-0.45);
						\draw (12.3,-0.75) -- (12.3,-0.45);
						\draw (9,-1.2) -- (11.3,-1.2) node[anchor=west] {$r_{j}$};
						\draw (9,-1.05) -- (9,-1.35);
						\draw (11.3,-1.05) -- (11.3,-1.35)  ;
					\end{scope}
					\endpgfgraphicnamed
				\end{tikzpicture}
				\caption{$a_{i}$ appears on the join and $r_{j}$ is subword of $a_{i}$ except at the end.}
				\label{fig:non_reduced_rep_forb}
			\end{center}
		\end{figure}
		2) Now, suppose there exists a forbidden word on the join of $w$ and $r_{k}$, and let $a_{i}$ be the first such occurrence. That means, there exists $0 < t \in (a_{i}, r_{k})^{|r_{k}|}$ such that $w$ is counted in $f_{a_{i}}(n+t)$. These terms constitute the third term on the RHS of~\eqref{eq:non_reduced_step2_final}.
		Further, since the forbidden word $a_{i}$ on the join may contain some repeated word as its subword, we need to carefully handle the multiple counting of the words in such situations. As explained earlier, it is enough to look at the situations when $r_{j}$ is a subword of $a_{i}$, except at the end. There are only two possible positions where $r_{j}$ (as a subword of $a_{i}$) can appear in $w r_{k}$, which contributes extra to the counting. These are depicted in Figure~\ref{fig:non_reduced_rep_forb}. The figure on the left represents the situation when $r_{j}$ appears before the join, as a subword of $w$, and the figure on the right represents the situation when $r_{j}$ appears on the join.\\
		i) In the first one (figure on the left), $r_{j}$ appears before the join as a subword of $a_{i}$ except at the end, and also as a subword of $w$ if and only if there exists $\alpha \in (a_{i}, r_{j})$ such that $\alpha > |r_{j}|$ and $\alpha \ge t+|r_{j}|$, where $0 < t \in (a_{i}, r_{k})$ is fixed. In these cases, $r_{j}$ contributes $m_{j}$ to the multiplicity of $w$ in $f(n)$. Further, $w$ is counted in the third term on the RHS of~\eqref{eq:non_reduced_step2_final}, but $r_{j}$ as a subword of $a_{i}$ does not contribute to the counting of $w$ in $f_{a_{i}}(n+t)$. Therefore, to balance the counting, we need to add the term $(m_{j}-1)f_{a_{i}}(n+t)$ for appropriate choices of $m_{j}$, determined using $\alpha$. These terms are included in the last term on the RHS of~\eqref{eq:non_reduced_step2_final}. \\
		%
		%
		ii) In the second one (figure on the right), $r_{j}$ occurs on the join. This holds if and only if there exists $\alpha \in (a_{i}, r_{j}) $ such that $\alpha > |r_{j}|$ and $\alpha > t$. Note that $\alpha > t$ is implicit here, for if $\alpha \le t$, then $r_{j}$ is a subword of $r_{k}$, which is not allowed since $\R$ is a reduced collection. Let us suppose $w$ gets counted only once in $f(n)$. Moreover, $w$ gets counted $m_{j}-1$ times in $\frac{g_{r_{j}}(n+s)}{m_{j}}$ for a suitable choice of $0 < s \in (r_{j}, r_{k})$ and once in $f_{a_{i}}(n+t)$. In this case as well, we need to add $(m_{j}-1)f_{a_{i}}(n+t)$, which are included in the last term on the RHS of~\eqref{eq:non_reduced_step2_final} for appropriate choices of $\alpha$.
		
		These two ((i) and ii)) together give us: $\# \{ \alpha \in (a_{i}, r_{j}) : \alpha > |r_{j}| \text{ and } \alpha \ge t+|r_{j}|\} = \# \{ \alpha \in (a_{i}, r_{j}) : \alpha > |r_{j}| \} = \gamma(a_{i}, r_{j})$. Therefore the terms obtained from these two cases constitute the last term in the RHS of~\eqref{eq:non_reduced_step2_final}.
		
		Now,~\eqref{eq:non_reduced2} is obtained by multiplying~\eqref{eq:non_reduced_step2_final} by $z^{-n}$ and taking sum over $n \ge 0$.
		
		Finally, consider $wa_{k}$, where $w$ is counted in $f(n)$, and $a_{k} \in \mathcal{F}$. Using the arguments developed so far in this proof and the proof of Theorem~\ref{thm:reduced_gen_fun}, we can obtain an expression for $f(n)$ given by
		\begin{align}\label{eq:non_reduced_step3_final}
			f(n) & =  \sum\limits_{i=1}^{s} \ \sum\limits_{0 \,< \, t \, \in \, (a_{i}, a_{k})} f_{a_{i}}(n+t)\  -  \ \sum\limits_{j = 1}^{\ell} \ \sum\limits_{ 0 \, < \, s \, \in \, (r_{j}, a_{k})^{|r_{j}|-1}} \left( 1 - \frac{1}{m_{j}} \right) g_{r_{j}} (n+s) \nonumber \\
			& \  \qquad + \ \sum\limits_{i=1}^{s} \ \sum\limits_{0 \, < \, t \, \in \, (a_{i}, a_{k})} \ \left[  \sum\limits_{j = 1}^{\ell} (m_{j} -1) \, \gamma_{t}(a_{i}, r_{j})  \right] f_{a_{i}}(n+t),
		\end{align}
		%
		%
		where $\gamma_{t}(a_{i}, r_{j}) = \# \{ \alpha \in (a_{i}, r_{j}) : \alpha> \max \{|r_{j}|, t\}\} $. Here the condition on $\alpha$ is obtained as $\alpha> \max \{|r_{j}|, t\}$, for each fixed $0 < t \in (a_{i}, a_{k})$. Here $\alpha > t$ needs to be imposed explicitly unlike the previous case since, for $\alpha \le t$, $a_{k}$ can contain $r_{j}$ as a subword, however it does not contribute to the counting. To take care of this extra condition, we define $\gamma_{t}(a_{i}, r_{j})$ as above and obtain Equation \eqref{eq:non_reduced_step3_final}. Finally to complete the proof, we multiply the above equation by $z^{-n}$ and take sum over $n \ge 0$ to obtain~\eqref{eq:non_reduced3}.

	\end{proof}

%
	
\end{document}